\documentclass[a4paper, oneside]{article}
\usepackage[usenames, dvipsnames]{color}
\usepackage{booktabs}
\definecolor{light-gray}{gray}{0.85}
\usepackage{graphicx}
\usepackage{tikz}
\usetikzlibrary{arrows.meta}
\usetikzlibrary{calc}
\usetikzlibrary{shapes.misc}
\usepackage{enumitem}
\usepackage{caption}
\usepackage{subcaption}
\usepackage{changepage}
\usepackage[draft]{hyperref}
\newcommand*{\Cdot}{\raisebox{-0.25ex}{\scalebox{1.25}{$\cdot$}}}
\usepackage[english]{babel}
\usepackage[utf8]{inputenc}
\usepackage{amssymb,amsmath, amsthm, mathtools}
\theoremstyle{Theorem}
\newtheorem{lemma}{Lemma}[section]
\newtheorem{prop}{Proposition}[section]
\newtheorem{theorem}{Theorem}[section]
\newtheorem{cor}{Corollary}[section]

\DeclareMathOperator*{\argmax}{arg\,max}

\newtheorem{definition}{Definition}[section]
\usepackage{thmtools}

\declaretheorem[style=definition,qed=$\square$,numberwithin=section]{example}

\theoremstyle{definition}
\newtheorem*{remark}{Remark}
\definecolor{hall}{gray}{0.6}
\usepackage{fullpage}

\usepackage{pgf}
\usepackage{tikz}
\usepackage[maxbibnames=99,style=numeric-comp,backend=bibtex,firstinits=true,sortcites=true,sorting=none]{biblatex}
\renewbibmacro{in:}{}
\usetikzlibrary{arrows,automata}
\usepackage{kbordermatrix}
\usepackage{soul}
\def\X{{\cal X}}
\def\Y{{\cal Y}}
\def\Z{{\cal Z}}
\def\B{{\cal B}}
\def\DEF{\stackrel{\mbox{\scriptsize{\normalfont{def}}}}{=}}
\makeatletter
\newsavebox\myboxA
\newsavebox\myboxB
\newlength\mylenA
\usepackage{titlesec}
\usepackage[titletoc,toc,title]{appendix}
\usepackage{colonequals}
\makeatletter
\def\namedlabel#1#2{\begingroup
   \def\@currentlabel{#2}%
   \label{#1}\endgroup
}
\makeatother
\mathtoolsset{centercolon}

\usepackage{accents}

\usepackage{authblk}
\begin{document}

\title{Existence of infinite Viterbi path for pairwise Markov models}
\date{}
\author{Jüri Lember}
\author{Joonas Sova}
\affil{\small University of Tartu, Liivi 2 50409, Tartu, Estonia.\\
Email: \textit{jyril@ut.ee}; \textit{joonas.sova@ut.ee}}
\maketitle

\begin{abstract} \noindent For hidden Markov models one of the most popular estimates of the hidden chain is the Viterbi path -- the path maximising the posterior probability. We consider a more general setting, called the pairwise Markov model, where the joint process consisting of finite-state hidden regime and observation process is assumed to be a Markov chain. We prove that under some conditions it is possible to extend the Viterbi path to infinity for almost every observation sequence which in turn enables to define an infinite Viterbi decoding of the observation process, called the Viterbi process. This is done by constructing a block of observations, called a barrier, which ensures that the Viterbi path goes trough a given state whenever this block occurs in the observation sequence.  \end{abstract}
\section{Introduction and preliminaries}
\subsection{Introduction}
We consider a Markov chain $Z =\{Z_k \}_{k\geq 1}$ with product
state space $\mathcal{X}\times \mathcal{Y}$, where $\mathcal{Y}$ is
a finite set (state space) and $\mathcal{X}$ is an arbitrary
separable metric space (observation space). Thus, the process $Z$
decomposes as $Z=(X,Y)$, where $X=\{X_k \}_{k\geq 1}$ and $Y=\{Y_k
\}_{k\geq 1}$ are random processes taking values in  $\mathcal{X}$
and $\mathcal{Y}$, respectively. The process $X$ is identified as an
observation process and the process $Y$, sometimes called the {\it
regime}, models  the observations-driving hidden state sequence.
Therefore our general model contains many well-known stochastic
models as a special case: hidden Markov models (HMM), Markov
switching models, hidden Markov models with dependent noise and many
more.  The {\it segmentation} or {\it path estimation} problem
consists of estimating the realization of $(Y_1,\ldots,Y_n)$ given a
realization $x_{1:n}$ of $(X_1,\ldots,X_n)$. A standard estimate is
any path $v_{1:n}\in \mathcal{Y}^n$ having maximum posterior
probability:
$$v_{1:n}=\argmax_{y_{1:n}}P(Y_{1:n}=y_{1:n}|X_{1:n}=x_{1:n}).$$
Any such  path is called {\it Viterbi path} and we are interested in
the behaviour of $v_{1:n}$ as $n$ grows. The study of asymptotics of
Viterbi path is complicated by the fact that adding one more
observation, $x_{n+1}$ can change the whole path, and so it is not
clear, whether there exists a limiting infinite Viterbi path. In
fact, as we show in Example \ref{noInf}, for some models the Viterbi
path keeps changing a.s. and so there is no infinite path. The goal
of the present paper is to establish the conditions that ensure the
existence of infinite Viterbi path, a.s. When this happens, one can
define infinite Viterbi decoding of $X$-process called {\it Viterbi
process}.
 In this paper, we construct the infinite Viterbi path using the
 barriers. A {\it barrier}  is a fixed-sized block in the observations $x_{1:n}$ that  fixes the Viterbi path up to
 itself: for every continuation of $x_{1:n}$, the Viterbi path up to
 the barrier remains unchanged. Therefore, if
almost every realization $x_{1:\infty}$ of $X$-process contains
infinitely many barriers, then the infinite Viterbi path exists
a.s. The main task of the paper is to exhibit the conditions (in
terms of the model) that guarantee the existence of infinite many
barriers, a.s. Having infinitely many barriers is not necessary for
existence of infinite Viterbi path (see Example \ref{Ex2}), but the
barrier-construction has several advantages. One of them is that it
allows to construct the infinite path {\it piecewise}, meaning that
to determine the first $k$ elements $v_{1:k}$ of the infinite path
it suffices to observe $x_{1:n}$ for $n$ big enough. Another great
advantage of the barriers is that under piecewise construction the
Viterbi process is typically a regenerative process. The
regenerativity allows to easily prove limit theorems to understand
the asymptotic behaviour of inferences based on Viterbi paths.

Our main construction theorems (Theorems \ref{th1} and \ref{thLSC})
generalize the piecewise construction in \cite{AVT4,AVT5}, where the
existence of Viterbi process were proven for HMM's.
 The important special case
of HMM is analysed in Subsection \ref{HMM}, but let us stress
that generalization beyond the HMM is far form being
straightforward. Moreover, we see that some assumptions of previous
HMM-theorem in \cite{AVT4} can be relaxed and the statements can be strengthened.   

The paper is organized as follows. In Subsection \ref{PMC-model}, we
introduce our model and some necessary notation; in Subsection
\ref{sec:VP}, the segmentation problem, infinite Viterbi path,
barriers and many other concepts are introduced and defined.  Also
the idea of piecewise construction is explained in detail. This
subsection also contains several examples like the above-mentioned
example of an HMM with no infinite Viterbi path (Example
\ref{noInf}). The subsection ends with the overview about the
history of the problem. In Section \ref{sec2} and \ref{sec3}, the main
barrier-construction theorems, Theorems \ref{th1} and \ref{thLSC},
are stated and proven. In Section \ref{sec:examples}, these theorems
are applied for several special cases and models: HMM's (Subsection
\ref{HMM}, discrete ${\cal X}$ (Subsection \ref{disc}) and linear
Markov switching model (Subsection \ref{sec:LMSW}).

\subsection{Pairwise Markov model} \label{PMC-model}
Let the observation-space $\mathcal{X}$ be a
separable metric space equipped with its Borel $\sigma$-field
$\B(\X)$. Let the state-space be $\mathcal{Y}=\{1,2,\ldots,|\mathcal{Y}|\}$, where $|\Y|$ is some
positive integer. We denote $\mathcal{Z}= \X \times \Y$, and equip
$\Z$ with product topology $\tau \times 2^\Y$, where $\tau$ denotes
the topology induced by the metrics of $\X$. Furthermore, $\Z$ is
equipped with its Borel $\sigma$-field $\B(\Z)=\B(\X) \otimes 2^\Y$,
which is the smallest $\sigma$-field containing sets of the form $A
\times B$, where $A \in \B(\X)$ and $B \in 2^\Y$. Let $\mu$ be a
$\sigma$-finite measure on $\B(\X)$ and let $c$ be the counting
measure on $2^\mathcal{Y}$. Finally, let
\begin{align*}
q \colon \Z^2 \rightarrow \mathbb{R}_{\geq 0}, \quad (z,z') \mapsto q(z|z')
\end{align*}
be a such a measurable non-negative function that for each $z' \in \mathcal{Z}$ the function $z \mapsto q(z|z')$ is a density with respect to product measure $\mu \times c$.

We define random process $Z =\{Z_k \}_{k\geq 1} = \{(X_k,Y_k) \}_{k \geq 1}$ as a homogeneous Markov chain on the two-dimensional space $\mathcal{Z}$ having the transition kernel density $q(z|z')$.
This means that the transition kernel of $Z$ is defined as follows:
\begin{align*}
&P(Z_2 \in A|Z_1=z')= \int_A q(z|z') \, \mu \times c (dz), \quad z' \in \mathcal{Z}, \quad A \in \B(\Z).
\end{align*}
The marginal processes $\{X_k \}_{k \geq 1}$ and $\{Y_k \}_{k \geq
1}$ will be denoted with $X$ and $Y$, respectively. Following
\cite{pairwise,pairwise2,pairwise3}, we call the process $Z$ a
\textit{pairwise Markov model} (PMM). It should be noted that even though $Z$ is a Markov chain, this doesn't necessarily imply that either of the marginal processes $X$ and $Y$ are Markov chains. However, it is not difficult to see that conditionally, given $Y$, $X$ is Markov
chain, and vice-versa \cite{pairwise}.

The letter $p$ will be used to denote the various joint and conditional densities. By abuse of notation, the corresponding probability law is
indicated by arguments of $p(\cdot)$, with lower-case $x_k$, $y_k$ and $z_k$ indicating random variables $X_k$, $Y_k$ and $Z_k$, respectively. For example
\begin{align*}
p(x_{2:n}, y_{2:n}|x_1,y_1)= \prod_{k=2}^n q(x_k, y_k|x_{k-1},y_{k-1}),
\end{align*}
where $x_{2:n}=(x_2, \ldots,x_n)$ and $y_{2:n}=(y_2, \ldots,y_n)$. Sometimes it is convenient to use other symbols beside $x_k,y_k,z_k$ as the arguments
of some density; in that case we indicate the corresponding probability law using the equality sign, for example
\begin{align*}
p(x_{2:n}, y_{2:n}|x_1=x,y_1=i)=q(x_2,y_2|x,i)\prod_{k=3}^n q(x_k, y_k|x_{k-1},y_{k-1}), \quad n \geq 3.
\end{align*}
Also $p(z_1)=p(x_1,y_1)$ denotes the initial distribution density of $Z$ with respect to measure $\mu_1 \times c$, where $\mu_1$ is some $\sigma$-finite measure on $\B(\X)$. Thus the joint density of $Z_{1:n}$ is
$p(z_{1:n})=p(z_1)p(z_{2:n}|z_1)$. For every $n \geq 2$ and $i,j \in \Y$ we also denote
\begin{align} \label{pij}
p_{ij}(x_{1:n})=\max_{y_{1:n} \colon y_1=i, y_n=j} \prod_{k=2}^n q(x_k,y_k|x_{k-1},y_{k-1}), \quad x_{1:n} \in \X^n.
\end{align}
Thus
\begin{align*}
p_{ij}(x_{1:n}) =\max_{y_{1:n} \colon y_1=i, y_n=j}p(x_{2:n},y_{2:n}|x_1,y_1).
\end{align*}

If $p(y_2|x_1,y_1)$ doesn't depend on $x_1$, and
$p(x_2|y_2,x_1,y_1)$ doesn't depend on neither $x_1$ nor $y_1$, then
$Z$ is called a \textit{hidden Markov model} (HMM). In that case, denoting
\begin{align*}
&p_{ij}=p(y_2=j|y_1=i),\quad f_{j}(x)=p(x_2=x|y_2=j),
\end{align*}
the transition kernel density factorizes into
\begin{align*} 
q(x, j|x',i)&=p(x_2=x|y_2=j,x_1=x',y_1=i)p(y_2=j|x_1=x',y_1=i)=p_{ij}f_{j}(x).
\end{align*}
Density functions $f_{j}$ are also called the \textit{emission densities}. When $\X$ is discrete, then $f_j(x)=P(X_2=x|Y_2=j)$ is called the \textit{emission probability} of $x$ from state $j$.

If $p(y_2|x_1,y_1)$ doesn't depend on $x_1$, and $p(x_2|y_2,x_1,y_1)$ doesn't depend on $y_1$, then following \cite{HMMbook} we call
$Z$ a \textit{Markov switching model}. Thus HMM's constitute a sub-class of Markov switching models. In the case of Markov switching model, denoting
\begin{align*}
&f_{j}(x|x')=p(x_2=x|y_2=j,x_1=x'),
\end{align*}
the transition kernel density becomes
\begin{align*}
q(x, j|x',i)&=p_{ij}f_{j}(x|x').
\end{align*}
It is easy to confirm that in case of Markov switching model (and therefore also in case of HMM) $Y$ is a homogeneous Markov chain with transition matrix $(p_{ij})$. Most PMM's
used in practice fall into the class of Markov switching models (see e.g. \cite{HMMbook} and the references therein for the practical applications of Markov switching models).
Figure \ref{fig:PMC} depicts the dependence structure of HMM,
Markov switching model and the general PMM.
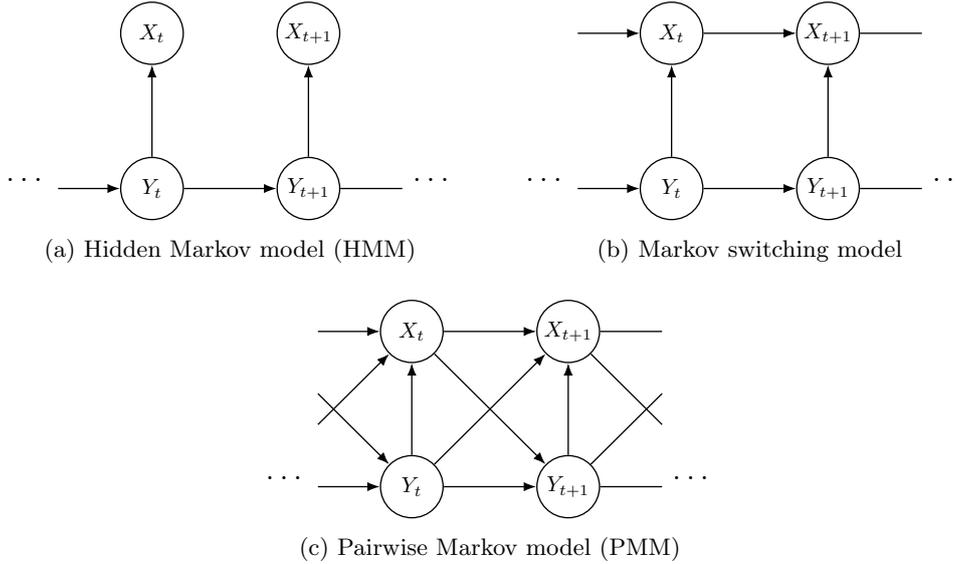
\begin{figure}
    \centering
    \begin{subfigure}[b]{0.4\textwidth}
    \resizebox{\linewidth}{!}{
\begin{tikzpicture}[auto,node distance=2.5cm,
                    semithick]
                    \tikzstyle{every state}=[minimum size=10mm,inner sep=0pt]
                   \tikzset{edge/.style = {->,> = latex'}}

  \node[state]         (Y1)              {$Y_t$};
  \node[state]         (Y2) [right of=Y1] {$Y_{t+1}$};
  \node[state]         (X1) [above of=Y1] {$X_t$};
  \node[state]         (X2) [above of=Y2] {$X_{t+1}$};
  \draw[-Latex] (Y1) -- (Y2);
  \draw[-Latex] (Y1) -- (X1);
  \draw[-Latex] (Y2) -- (X2);
 \draw[-Latex]  ++ (-1.5,0) --   (Y1);
  \draw (Y2) -- ++ (1.5,0);
\path (-2.5,0) -- (-1.5,0) node [font=\Large, midway, sloped] {$\dots$};
\path (4,0) -- (5,0) node [font=\Large, midway, sloped] {$\dots$};
\end{tikzpicture}
}
        \caption{Hidden Markov model (HMM)}
    \end{subfigure}
    \quad  
      \begin{subfigure}[b]{0.4\textwidth}
  \resizebox{\linewidth}{!}{ \begin{tikzpicture}[auto,node distance=2.5cm,
                    semithick]
                    \tikzstyle{every state}=[minimum size=10mm,inner sep=0pt]
                   \tikzset{edge/.style = {->,> = latex'}}

  \node[state]         (Y1)              {$Y_t$};
  \node[state]         (Y2) [right of=Y1] {$Y_{t+1}$};
  \node[state]         (X1) [above of=Y1] {$X_t$};
  \node[state]         (X2) [above of=Y2] {$X_{t+1}$};
  \draw[-Latex] (Y1) -- (Y2);
  \draw[-Latex] (Y1) -- (X1);
  \draw[-Latex] (Y2) -- (X2);
    \draw[-Latex] (X1) -- (X2);
 \draw[-Latex]  (-1.5,2.5) --   (X1);
 \draw[-Latex]  (-1.5,0) --   (Y1);
  \draw (Y2) -- ++ (1.5,0);
    \draw (X2) -- ++ (1.5,0);
\path (-2.5,0) -- (-1.5,0) node [font=\Large, midway, sloped] {$\dots$};
\path (4,0) -- (5,0) node [font=\Large, midway, sloped] {$\dots$};
\end{tikzpicture}
 }
        \caption{Markov switching model}
    \end{subfigure}
    
\bigskip
      \begin{subfigure}[b]{0.4\textwidth}
  \resizebox{\linewidth}{!}{ \begin{tikzpicture}[auto,node distance=2.5cm,
                    semithick]
                    \tikzstyle{every state}=[minimum size=10mm,inner sep=0pt]
                   \tikzset{edge/.style = {->,> = latex'}}

  \node[state]         (Y1)              {$Y_t$};
  \node[state]         (Y2) [right of=Y1] {$Y_{t+1}$};
  \node[state]         (X1) [above of=Y1] {$X_t$};
  \node[state]         (X2) [above of=Y2] {$X_{t+1}$};
  \draw[-Latex] (Y1) -- (Y2);
  \draw[-Latex] (Y1) -- (X1);
  \draw[-Latex] (Y2) -- (X2);
    \draw[-Latex] (Y1) -- (X2);
  \draw[-Latex] (X1) -- (Y2);
    \draw[-Latex] (X1) -- (X2);
 \draw[-Latex]  (-1.5,2.5) --   (X1);
 \draw[-Latex]  (-1.5,0) --   (Y1);
  \draw[-Latex]  (-1.5,1.5) --   (Y1);
    \draw[-Latex]  (-1.5,1) --   (X1);
  \draw (Y2) -- ++ (1.5,0);
    \draw (Y2) -- ++ (1.5,1.5);
        \draw (X2) -- ++ (1.5,-1.5);
    \draw (X2) -- ++ (1.5,0);
\path (-2.5,0) -- (-1.5,0) node [font=\Large, midway, sloped] {$\dots$};
\path (4,0) -- (5,0) node [font=\Large, midway, sloped] {$\dots$};
\end{tikzpicture}
 }        \caption{Pairwise Markov model (PMM)}
    \end{subfigure}
        \caption{Dependence graphs of different types of PMM's}\label{fig:PMC}
\end{figure}

\subsection{Viterbi path}\label{sec:VP} The {\it segmentation problem} in general
consists of guessing or estimating the unobserved realization of
process $Y_{1:n}$ -- the true path  -- given the realization
$x_{1:n}$ of the observation process $X_{1:n}$. Since the true path
cannot be exactly known,  the segmentation procedure merely consists
of finding the path that in some sense is the best approximation.
Probably the most popular estimate is the path with maximum
posterior probability. This path will be denoted with $v(x_{1:n})$
and also with $v_{1:n}$, when $x_{1:n}$ is assumed to be fixed:
$$v_{1:n}=v(x_{1:n})=\argmax_{y_{1:n}}p(y_{1:n},x_{1:n})=\argmax_{y_{1:n}}P(Y_{1:n}=y_{1:n}|X_{1:n}=x_{1:n}).$$
Typically $v_{1:n}$ is called Viterbi or MAP path (also Viterbi or
MAP alignment). Clearly $v_{1:n}$ might not be unique. As it is well
known, Viterbi path minimizes the average error over all possible
paths, when the error between two sequences is zero if they are
totally equal and one otherwise. On the other hand, Viterbi path is
not in general the one that minimizes the expected number of errors,
when the number of errors between two sequences are measured entry
by entry (Hamming metric). For more detailed discussion about the
segmentation problem and the properties of different estimates, we
refer to \cite{intech,seg,chris,Vrisk,peep}. Although these papers
deal with HMM's only, the general theory applies for any model
including PMM's.

The Viterbi path inherits its name by famous Viterbi algorithm that
is used to find the Viterbi path in the case of HMM. It is easy to
see that the algorithm also applies in the case of PMM. To see that,
denote for every $y\in {\cal Y}$
\begin{align*}
\delta_1(y)=p(x_1,y_1=y), \quad
\delta_t(y)=\max_{y_{1:t} \colon y_t=y}p(x_{1:t},y_{1:t}), \quad t\geq 2.
\end{align*}
Clearly $\delta_t(y)$ also depends on $x_{1:t}$, but in our case
the path $x_{1:n}$ is typically fixed and therefore $x_{1:t}$
is left out from the definition. The recursion behind the Viterbi
algorithm is now
\begin{align}\label{recursion}
\delta_{t+1}(y)=\max_{y'}\delta_t(y')q(x_{t+1},y|x_t,y'),\quad t=2,\ldots,n.
\end{align}
At each time $t+1$ and state $y$ the algorithm remembers the state
$y'$ achieving the maximum in \eqref{recursion} and by
backtracking from the state $v_n=\argmax_{y\in {\cal
Y}}\delta_n(y)$, the Viterbi path can be found. To avoid the
numerical underflow, the logarithmic or rescaled versions of the
Viterbi recursion can be used, see e.g. \cite{seg}.

Because Viterbi algorithm applies for PMM's as easily as for
HMM's, using Viterbi path in segmentation is appealing
computationally as well as conceptually. Therefore, to study the
statistical properties of Viterbi path-based inferences, one has to
know the long-run or typical behaviour of random vectors
$v(X_{1:n})$. As argued in \cite{AVT4}, behaviour of $v(X_{1:n})$ is
not trivial since the $(n+1)^{\mathrm{th}}$ observation can in
principle change the entire alignment based on the previous observations $x_{1:n}$. It might
happen with a positive probability that the first $n$ entries of
$v(x_{1:n+1})$ are all different from corresponding entries of
$v(x_{1:n})$. If this happens again and again, then the first
element of $v(x_{1:n})$ keeps changing as $n$ grows and there is not
such thing as limiting Viterbi path. On the other hand, it is
intuitively clear that there is a positive probability to observe
$x_{1:n}$
 such that regardless of the value of the $(n+1)^{\mathrm{th}}$ observation (provided $n$ is sufficiently large),
the paths $v(x_{1:n})$ and $v(x_{1:n+1})$ agree on first $u$
elements, where $u<n$. If this is true, then no matter what happens
in the future, the first $u$ elements of the paths remain constant.
Provided there  is an increasing unbounded sequence $u_i$
($u<u_1<u_2<\ldots$) such that the path up to $u_i$ remains
constant, one can define limiting or infinite Viterbi path. Let us
formalize the idea. In the following definition $v(x_{1:n})$ is a
Viterbi path and $v(x_{1:n})_{1:t}$ are the first $t$ elements of
the $n$-elemental vector $v(x_{1:n})$.

\begin{definition} Let $x_{1:\infty}$ be a realization of
$X$. The sequence $v_{1:\infty}\in {\cal Y}^{\infty}$ is
called \emph{infinite Viterbi path of $x_{1:\infty}$} if for any
$t \geq 1$ there exists $m(t)\geq t$ such that
\begin{equation}\label{infty}
v(x_{1:n})_{1:t}=v_{1:t},\quad \forall n\geq m(t).\end{equation}
\end{definition}

Hence $v_{1:\infty}$ is the infinite Viterbi path of
$x_{1:\infty}$ if for any $t$, the first $t$ elements of
$v_{1:\infty}$ are the first $t$ elements of a Viterbi path
$v(x_{1:n})$ for all $n$ big enough $(n\geq m(t))$. In other words,
for every $n$ big enough, there exists at least one  Viterbi path so
that $v(x_{1:n})_{1:t}=v_{1:t}$. Note that above-stated definition
is equivalent to the following: for every $t \geq 1$,
\begin{equation}\label{koond}
\lim_n v(x_{1:n})_{1:t}\to v_{1:t}.
\end{equation}
Indeed, since ${\cal Y}^t$ is finite, the convergence (\ref{koond})
holds if and only if $v(x_{1:n})_{1:t}=v_{1:t}$ eventually, and this
is exactly (\ref{infty}). For infinite ${\cal Y}$, (\ref{infty}) is
obviously much stronger than (\ref{koond}), and in this case, the
infinite Viterbi path is defined via (\ref{koond}), see
\cite{Ritov}. 

As we shall see, for many PMM's the infinite
Viterbi path exists for almost every realization of $X$. However,
the following counterexample shows that for some models the infinite
Viterbi path exists for almost no realization of $X$.

\begin{example} \label{noInf} Let $p \in (\frac{1}{2},1)$; then there exists positive integer $K $
such that taking $\epsilon=\frac{1}{2K}$, we have
\begin{align}\label{vor1}
\epsilon< 1-p-\frac{\epsilon}{2} < \dfrac{1}{2}< p-\frac{\epsilon}{2} < 1.
\end{align}
We look at the model where $\mathcal{X}=\{1,2\}$, $\mathcal{Y}=\{1,2,\ldots,K+2\}$ and the transmission matrix of $Z$ is
\begingroup
\renewcommand*{\arraystretch}{1.3}
\renewcommand{\kbldelim}{(}
\renewcommand{\kbrdelim}{)}
\[
  \kbordermatrix{
          & (1,1) & (2,1) & (1,2) & (2,2) & (1,3) & (2,3) & \ldots & (1,K+2) & (2,K+2) \\
    (1,1) & p-\frac{\epsilon}{2} & 1-p-\frac{\epsilon}{2} & 0 & 0 & \epsilon^2 & \epsilon^2 & \ldots & \epsilon^2 & \epsilon^2 \\
    (2,1) & p-\frac{\epsilon}{2} & 1-p-\frac{\epsilon}{2} & 0 & 0 & \epsilon^2 & \epsilon^2 &  \ldots & \epsilon^2 & \epsilon^2 \\
    (1,2) & 0 & 0 & 1-p-\frac{\epsilon}{2} &  p-\frac{\epsilon}{2} & \epsilon^2 & \epsilon^2 & \ldots & \epsilon^2 & \epsilon^2  \\
    (2,2) & 0 & 0 & 1-p-\frac{\epsilon}{2} &  p-\frac{\epsilon}{2} & \epsilon^2 & \epsilon^2 & \ldots & \epsilon^2 & \epsilon^2 \\
    (1,3) & 0 & 0 & 0 & 0 & \epsilon & \epsilon & \ldots & \epsilon & \epsilon \\
    (2,3)  & 0 & 0 & 0 & 0 & \epsilon & \epsilon & \ldots & \epsilon & \epsilon \\
  \vdots  & \vdots  & \vdots  & \vdots  & \vdots  & \vdots & \vdots & \ddots & \vdots & \vdots  \\
    (1,K+2) & 0 & 0 & 0 & 0 & \epsilon & \epsilon & \ldots & \epsilon & \epsilon \\
    (2,K+2)  & 0 & 0 & 0 & 0 & \epsilon & \epsilon & \ldots & \epsilon & \epsilon \\
  }.
\]
\endgroup
We assume that the initial distribution of $Z$ is such that $$P(X_1=1)=P(X_1=2)=P(Y_1=1)=P(Y_1=2)=\frac{1}{2}$$ and $X_1$ and $Y_1$ are independent.

Let's see now what are the possible Viterbi paths for some observation sequence $x_{1:n}$. We have
\begin{align} \label{pmax}
\max_{y_{1:n}}p(x_{1:n},y_{1:n})
&= \dfrac{1}{4} \max_{y_1 \in \{1,2\}, \: y_{2:n} \in  \mathcal{Y}^{n-1} } \prod_{k=2}^n q(x_k,y_k|x_{k-1},y_{k-1}) \notag\\
&=\dfrac{1}{4}\max_{y_{1:n} \in  \{1,2\}^n} \prod_{k=2}^n q(x_k,y_k|x_{k-1},y_{k-1}) \notag\\
&= \dfrac{1}{4}\max_{y_{1:n} \in  \{(1,\ldots,1),(2,\ldots,2)\}} \prod_{k=2}^n q(x_k,y_k|x_{k-1},y_{k-1}),
\end{align}
where the second equality holds by (\ref{vor1}). Let $n_i(x_{2:n})$ be the number of $i$-s in $x_{2:n}$. Since
\begin{align*}
\prod_{k=2}^n q(x_k,1|x_{k-1},1)=\left(p-\frac{\epsilon}{2}\right)^{n_1(x_{2:n})}\left(1-p-\frac{\epsilon}{2}\right)^{n_2(x_{2:n})}
\end{align*}
and
\begin{align*}
\prod_{k=2}^n q(x_k,2|x_{k-1},2)=\left(p-\frac{\epsilon}{2}\right)^{n_2(x_{2:n})}\left(1-p-\frac{\epsilon}{2}\right)^{n_1(x_{2:n})}
\end{align*}
we have by the fact that $p > 1- p$ and (\ref{pmax}), that Viterbi path of $x_{1:n}$ can be expressed as
\begin{align} \label{vitX}
v(x_{1:n})=\begin{cases}(1,\ldots,1), & \mbox{if $n_1(x_{2:n})\geq n_2(x_{2:n})$}\\
(2,\ldots,2), & \mbox{else}
\end{cases}.
\end{align}

Now, let's take a closer look at the behaviour of $Z$. Note that at some time, let's say at time $T$, $Y$ moves to state space $\{ 3,\ldots,K+2 \}$.
Time $T$ is a random variable that is almost surely finite. Before time $T$, $Y$ is constantly in the state $1$ or constantly in the state $2$
(both possibilities having probability $\frac{1}{2}$). After and at time $T$, $Y$ is always in $\{ 3,\ldots,K+2 \}$. Note that $\{X_{T+k} \}_{k \geq 0}$ is an i.i.d. Bernoulli sequence with parameter $\frac{1}{2}$. Let
\begin{align*}
S_k=n_1(X_{2:k})-n_2(X_{2:k}), \quad k\geq 2.
\end{align*}
Random process $\{S_{T+k} \}_{k \geq 0}$ is a simple symmetric random walk with random starting point. Therefore the process $\{S_k \}_{k \geq 2}$ will almost surely fall below zero i.o. and rise above zero i.o. Together with (\ref{vitX}) this implies that almost no realization of $X$ has an infinite Viterbi path.

It is easy to confirm that this model is HMM with
\begin{align*}
&p_{11}=p_{22}=1-\epsilon, \quad p_{12}=p_{21}=0,\\
&p_{1k}=p_{2k}=\dfrac{\epsilon}{K}, \quad p_{k1}=p_{k2}=0, \quad k=3,\ldots,K+2,\\
&p_{kl}=\dfrac{1}{K}, \quad  k,l \in \{3,\ldots,K+2\},
\end{align*}
and
\begin{align*}
&f_{2}(2)=f_{1}(1)=\dfrac{1}{1-\epsilon}\left(p-\dfrac{\epsilon}{2} \right), \quad f_{1}(2)=f_{2}(1)=\dfrac{1}{1-\epsilon}\left(1-p-\dfrac{\epsilon}{2} \right),\\
&f_k(1)=f_k(2)=\dfrac{1}{2}, \quad k=3,\ldots,K+2.
\end{align*}
\end{example}

\paragraph{Nodes.} Suppose now $x_{1:\infty}$ is such that infinite
Viterbi path exists. It means that for every time $t$, there exists
time $m(t)\geq t$ such that the first $t$ elements of $v(x_{1:n})$
are fixed as soon as $n\geq m$. Note that if $m(t)$ is such a time,
then $m(t)+1$ is such a time too. Theoretically, the time $m$ might
depend on the whole sequence $x_{1:\infty}$. This means that after
observing the sequence $x_{1:m}$, it is not yet clear, whether the
first $t$ elements of Viterbi path are now fixed (for any
continuation of $x_{1:m}$) or not. In practice, one would not like
to wait infinitely long, instead one prefers to realize that the
time $m(t)$ is arrived right after observing $x_{1:m}$. In this
case, the (random) time $m(t)$ is the stopping time with respect to
the observation process. In particular, it means the following: for
every possible continuation $x_{m+1:n}$ of $x_{1:m}$, the Viterbi
path at time $t$ passes the state $v_t$, let that state be $i$. This
requirement is fulfilled, when the following holds: for every two states
$j,k\in {\cal Y}$
\begin{equation}\label{r-node}
\delta_t(i)p_{ij}(x_{t:m})\geq \delta_t(k)p_{kj}(x_{t:m}),
\end{equation}
where $p_{ij}(\cdot)$ is defined in (\ref{pij}). Indeed,
 there might be several states satisfying
(\ref{r-node}), but the ties can always be broken in favour of the
state $i$, so that whenever $n \geq m$, there is at least one
Viterbi path $v(x_{1:n})$ that passes the state $i$ at time $t$.
 Therefore, if at time $t$, there is a state  $i$ satisfying (\ref{r-node}), then $m$ is the time
$m(t)$ required in (\ref{infty}) and it depends on $x_{1:m}$ only.
\begin{definition} Let $x_{1:m}$ be a vector of observations.  If equalities (\ref{r-node}) hold for any pair of states $j$ and $k$, then the
time  $t$  is called  an {\rm $i$-node of order $r=m-t$.}
Time $t$ is called a {\rm strong $i$-node of order $r$}, if it is an $i$-node of order $r$, and
the inequality (\ref{r-node}) is strict for any $j$ and $k\ne i$ for which the left side of the inequality is positive. We
call $t$ a node of order $r$ if for some $i$, it is an $i$-node of
order $r=m-t$.
\end{definition}
The definition of node is a straightforward generalization of the
corresponding definition in \cite{AVT4,AVT5,K2}. Note that when
$t$ is a node of order $r$, then ${t-1}$ is a node of order
$r+1$.
\begin{example}\label{Ex2} Following is an example of a model, for which infinite Viterbi path always exists, but no nodes ever occur. Let $Z$ be a HMM with ${\cal Y}={\cal
X}=\{1,2\}$,  transition matrix of $Y$ being identity. Suppose $p \in (\frac{1}{2},1)$ and let emission probabilities be
\begin{align*}
f_{1}(1)=f_2(2)=p, \quad f_1(2)=f_2(1)=1-p.
\end{align*}
Let the initial distribution be uniform. This trivial model picks parameter
$p$ or $1-p$ with probability ${1\over 2}$ and then an i.i.d.
Bernoulli sample with chosen probability. Let, for any $x_{1:t}\in
\{0,1\}^t$, $n_i(x_{1:t})$ be the number of $i$-s in $x_{1:t}$. Now clearly for any $n \geq 1$
$$v(x_{1:n})=\begin{cases}
                 (1,\ldots,1), & \hbox{if $n_1(x_{1:t})\geq n_2(x_{1:t})$} \\
                 (2,\ldots,2), & \hbox{else}
               \end{cases}.
$$ Since by SLLN
$${n_2(X_{1:n})\over n}\to \begin{cases}
                               p, & \hbox{if $Y_1=2$} \\
                               1-p, & \hbox{if $Y_1=1$}
  \end{cases} \quad \mbox{a.s.,}
$$
we see that for almost every realization of $X$ the
infinite Viterbi path exists. This infinite path is constantly 1 if
$Y_1=1$ and constantly 2 if $Y_1=2$. Surely, for any $t$, there
exists $m(t)$ such that (\ref{infty}) holds, but in this case $m(t)$
depends on the whole sequence $x_{1:\infty}$, because for any
$x_{1:m}$ one can find a continuation $x_{m+1:n}$ such that
$v(x_{1:t})\ne v(x_{1:n})_{1:t}$. This implies that there cannot
be any nodes in any sequence $x_{1: \infty}$. Indeed, for any $x_{t+1:m}$, it
holds $p_{12}(x_{t:m})=p_{21}(x_{t:m})=0,$ and so inequalities
(\ref{r-node}) cannot hold.
\end{example}
\paragraph{Barriers.}
The goal of the present paper is to find sufficient conditions for
almost every realization of observation process to have infinitely many nodes. Whether a time
$t$ is a node of order $r$ or not depends, in general, on the
sequence $x_{1:t+r}$. Sometimes, however, there is some small
block of observations that guarantees the
existence of a node regardless of the other observations. Let
us illustrate this by an example.

\begin{example} Suppose that   there exists a state $i\in {\cal Y}$
such that for any triplet $y_{t-1},y_t,y_{t+1} \in {\cal Y}$
\begin{equation}\label{tr-barrier}
q(x_t,i|x_{t-1},y_{t-1})q(x_{t+1},y_{t+1}|x_t,i)\geq q(x_t,y_t|x_{t-1},y_{t-1})q(x_{t+1},y_{t+1}|x_t,y_t).
\end{equation}
Then
\begin{align*}
\delta_t(i)q(x_{t+1},y_{t+1}|x_t,i)&=\max_{y'}\delta_{t-1}(y')q(x_t,i|x_{t-1},y')q(x_{t+1},y_{t+1}|x_t,i)\\
&\geq
\max_{y'}\delta_{t-1}(y')q(x_t,y_t|x_{t-1},y')q(x_{t+1},y_{t+1}|x_t,y_t)\\
&=\delta_t(y_t)q(x_{t+1},y_{t+1}|x_t,y_t).\end{align*}
We thus have that $t$ is an
$i$-node of order 1, because for every pair $j,k \in \Y$
$$\delta_t(i)p_{ij}(x_t,x_{t+1})\geq
\delta_t(k)p_{kj}(x_t,x_{t+1}).$$ Whether (\ref{tr-barrier}) holds
or not, depends on triplet $(x_{t-1},x_t,x_{t+1})$. In case of Markov switching model, (\ref{tr-barrier}) is
\begin{align*}
p_{y_{t-1}i}f_{i}(x_t|x_{t-1}) \cdot p_{i y_{t+1}} f_{y_{t+1}}(x_{t+1}|x_t) \geq p_{y_{t-1}y_t}f_{y_t}(x_t|x_{t-1}) \cdot p_{y_t y_{t+1}} f_{y_{t+1}}(x_{t+1}|x_t).
\end{align*}
And in a more special case of HMM, (\ref{tr-barrier}) is equivalent
to
\begin{align}\label{1-barHMM}
p_{y_{t-1}i}f_i(x_t) \cdot p_{i y_{t+1}} \geq p_{y_{t-1}y_t}f_{y_t}(x_t) \cdot p_{y_t y_{t+1}} .
\end{align}
\end{example}
The inequalities (\ref{1-barHMM}) have very clear meaning -- when
the observation $x_t$ has relatively big probability of  being
emitted from state $i$ (in comparison of being emitted from any
other state),  then regardless of the observations before or after
$x_t$, time $t$ is  a node. In particular, this is the case when the
supports of the emission distributions are different and $x_t$ can
be emitted from one state, only. On the other hand, for many models,
there are no such $x_t$ possible, so (\ref{1-barHMM}) is rather an
exception than a rule.
\begin{definition} Given $i\in \Y$, $b_{1:M}$  is called an
{\rm (strong) $i$-barrier of order $r$ and length $M$}, if, for any
$x_{1:\infty}$ with  $x_{m-M+1:m}=b_{1:M}$   for some $m\geq M$,
${m-r}$  is an (strong) $i$-node of order $r$.
\end{definition}
Hence, if (\ref{tr-barrier}) holds, then the triplet
$(x_{t-1},x_t,x_{t+1})$ is an $i$-barrier of order 1 and length 3.
In what follows, we give some sufficient conditions that guarantee
the existence of infinitely many barriers  in almost every
realization of $X$. More closely, we construct a set ${\cal
X}^*\subset{\cal X}^M$ such that every vector $x_{1:M}$ from ${\cal
X}^*$ is $i$-barrier of order $r$ for a given state $i\in \mathcal{Y}$, and $P(X \in \mathcal{X}^* \mbox{ i.o.})=1$, where
\begin{align*}
\{ X \in \mathcal{X}^* \mbox{ i.o.} \} \DEF \bigcap_{k=1}^\infty \bigcup_{l=k}^\infty \{X_{l:l+M-1} \in \mathcal{X}^* \} .
\end{align*}
Since every  barrier contains a $r$-order $i$-node, having infinitely many
barriers in $x_{1:\infty}$ entails infinitely many $i$-nodes of
order $r$, let the locations of these nodes be $u_1<u_2<\cdots$. Let
$m \geq u_2+r$.  There must exist a Viterbi path
$v(x_{1:m})$ passing state $i$ at time $u_1$. There also exists a
Viterbi path passing $i$ at time $u_2$. If $v(x_{1:m})$ is unique,
then the path passes $i$ at both times, but if Viterbi path is
not unique and  $u_1$ and $u_2$ are too close to each other, then
there might not be possible to break ties in favour of $i$ at $u_1$
and $u_2$ simultaneously, as is shown in the following example.

\begin{example} Let $Z$ be HMM with $ |\mathcal{Y}| \geq 4$ and let for some $\epsilon \in (0,1)$
\begin{align*}
&p_{ij}=\epsilon, \quad (i,j) \in 2^{\{1,2\}}=\{(1,1),(1,2),(2,1),(2,2)\},\\
&p_{13}=p_{24}=p_{41}=p_{32}=\epsilon, \quad p_{14}=p_{23}=p_{42}=p_{31}=0.
\end{align*}
Also, let $\mathcal{X}$ be finite, and let $1,2 \in \mathcal{X}$ be such that for some $p \in (0,1)$
\begin{align*}
&f_1(1)=f_2(1)=p; \quad f_k(1)=0, \quad k \in \mathcal{Y} \setminus \{1,2\};\\
&f_3(2)=f_4(2)=p; \quad f_k(2)=0, \quad k \in \mathcal{Y} \setminus \{3,4\}.
\end{align*}
Thus, whenever $X_t=1$, then $Y_t \in \{1,2\}$, and whenever $X_t=2$, then $Y_t \in \{3,4\}$ ($t\geq 2$). Note that the word $(1,1,2,1,1)$ is 1- and 2-barrier of length 5 and of order 3. To see this, note that when for some $t$, $x_{t:t+4}=(1,1,2,1,1)$ then for all $k,j \in \mathcal{Y}$
\begin{align*}
\delta_{t+1}(k)p_{kj}(1,2,1,1)=\delta_t(i^*)p^4\epsilon^4 \cdot \mathbb{I}_{\{1,2\}}(k)\mathbb{I}_{\{1,2\}}(j),
\end{align*}
where $i^*=\argmax_{l \in \{1,2\}}\delta_t(l)$ and $\mathbb{I}$ denotes the indicator function. Thus time $t+1$ is an $i$-node of order $3$ for $i \in \{1,2\}$: for all $k,j \in \Y$
\begin{align*}
\delta_{t+1}(i)p_{ij}(1,2,1,1)&=\delta_t(i^*)p^4\epsilon^4 \cdot \mathbb{I}_{\{1,2\}}(i)\mathbb{I}_{\{1,2\}}(j) \\
&\geq  \delta_t(i^*)p^4\epsilon^4 \cdot \mathbb{I}_{\{1,2\}}(k)\mathbb{I}_{\{1,2\}}(j)\\
&=\delta_{t+1}(k)p_{kj}(1,2,1,1).
\end{align*}
This means that the word $(1,1,2,1,1)$ is indeed a 1- and 2-node of order 3.

Similarly, $(2,1,1,1,1)$ is also 1- and 2-barrier of length 5 and of order 3. Assuming that $Y$ is irreducible, we have that the word $(1,1,2,1,1,1,1)$ occurs in $X$ infinitely many times. Suppose now that $t$ is such that $x_{t:t+6}$ is equal to that word. Then $t+1$ and $t+3$ are both 1- and 2-nodes. Now, breaking ties at these locations differently (to 1 at $t+1$ and to 2 at $t+3$, or vice-versa) is acceptable. But breaking ties to the same value (either both to 1 or both to 2) will result in a zero-likelihood path. Indeed, for $y_{t+1:t+3} \in \{1,2\} \times \Y \times \{1,2\}$
\begin{align*}
p(x_{t+2:t+3},y_{t+2:t+3}|x_{t+1},y_{t+1})&=p_{y_{t+1}y_{t+2}}f_{y_{t+2}}(2) \cdot p_{y_{t+2}y_{t+3}}f_{y_{t+3}}(1)\\
&=\epsilon^2 \mathbb{I}_{ \{(1,3,2) \} }(y_{t+1:t+3}) \cdot \mathbb{I}_{\{ (2,4,1) \}}(y_{t+1:t+3}) \cdot p^2.
\end{align*}
\end{example}

This problem does not occur, if the nodes are strong or if  $u_2
\geq u_1+r$. Indeed, since $u_1$ is an $i$-node of order $r$, then
by definition of $r$-order node, between times $u_1$ and $u_1+r+1$,
the ties can be broken so that whatever state the Viterbi path
passes at time $u_2$, it passes $i$ at time $u_1$. Thus, if the
locations of nodes  $u_1<u_2<\cdots$ are such that $u_k\geq
u_{k-1}+r$ for all $k\geq 2$, it is possible to construct the infinite Viterbi path so that it passes the state $i$ at every time
$u_k$. In what follows, when the nodes $u_k$ and $u_{k-1}$ are such
that $u_k\geq u_{k-1}+r$, then the nodes are called
\textit{separated}. Of course, there is no loss of generality in
assuming that the nodes $u_1<u_2<\cdots$ are separated, because from
any non-separated sequence of nodes it is possible to pick a separated
subsequence. Another approach is to enlarge the
barriers so that two barriers cannot overlap and, therefore, are
separated. This is the way barriers are defined in \cite{AVT5}.

\paragraph{Construction of infinite Viterbi path.}
Having infinitely many separated nodes $u_1<u_2<\cdots$ or
order $r$, it is possible to construct the infinite Viterbi path
{\it piecewise}. Indeed, we know that for every $n\geq u_k+r$, there is
a Viterbi path $v_{1:n}=v(x_{1:n})$ such that $v_{u_j}=i$, $j=1,\ldots,k$.
Because of that property  and by optimality principle  clearly the
piece $v_{u_{j-1}:u_j}$ depends on the observations
$x_{u_{j-1}:u_j}$, only. Therefore $v_{1:\infty}$ can be constructed in the
following way: first use the observations $x_{1:u_1}$ to find the
first piece $v_{1:u_1}$ as follows:
$$v_{1:u_1}=\argmax_{y_{1:u_1} \colon y_{u_1}=i}p(x_{1:u_1},y_{1:u_1}).$$
Then use  $x_{u_1:u_2}$ to find the second piece $v_{u_1:u_2}$ as
follows:
$$v_{u_1:u_2}=\argmax_{y_{u_{1}:u_2} \colon y_{u_1}=y_{u_2}=i }p(x_{u_1:u_2},y_{u_1:u_2}),$$
and so on. Finally use $x_{u_k:n}$ to find the last piece
$v_{u_k:n}$ as follows:
$$v_{u_k:n}=\argmax_{y_{u_{k}:n} \colon y_{u_k}=i}p(x_{u_k:n},y_{u_{k}:n}).$$
The last piece $v_{u_k:n}$ might change as $n$ grows, but the rest
of the Viterbi path is now fixed. Thus, if $x_{1:\infty}$ contains
infinitely many nodes, the whole infinite path
can be constructed piecewise.

If the nodes $u_k$ are strong (not necessarily
separated) then the piecewise construction detailed above is
achieved when the Viterbi estimation is done by a
lexicographic or co-lexicographic tie-breaking scheme induced by
some ordering on $\Y$. Indeed, since the $i$-nodes $u_k$ are strong,
we know that regardless of tie-breaking scheme $v(x_{1:n})_{u_k}=i$
for all $k \geq 1$ and $n \geq u_k+r$. Therefore the lexicographic
ordering ensures that for all $k \geq 1$ and $n \geq u_k+r$
\begin{align*}
v_{1:n}=\argmax_{y_{1:n}} p(x_{1:n},y_{1:n})=(\argmax_{y_{1:u_k}} p(x_{1:u_k},y_{1:u_k}), \argmax_{y_{u_{k}+1:n}} p(x_{u_k:n},y_{u_k}=i, y_{u_{k}+1:n})).
\end{align*}
This shows that $v(x_{1:n})_{1:u_k}$ is independent of $n \geq u_k+r$ for all $k \geq 1$ and so the infinite Viterbi path is well-defined.
\paragraph{Viterbi process.} The notion of infinite Viterbi path of a fixed realization $x_{1:\infty}$ naturally carries over to an infinite Viterbi path of $X$,
called the \textit{Viterbi process}. Formally,
this process is defined as follows.
\begin{definition} A random process $V=\{V_{k}\}_{k \geq 1}$ on space $\Y$ is called a \emph{Viterbi process}, if the event $\{V \mbox{ is not an infinite Viterbi path of }X \}$
is contained in a set of zero probability measure.
\end{definition}
If there exists a barrier set $\X^*$ consisting of $i$-barriers of fixed
order and satisfying $P(X \in \X^* \mbox{ i.o.})=1$, then the
Viterbi process can be constructed by applying the piecewise
construction detailed above to the process $X$. However, this construction has a
serious weakness: it requires that the ties are broken in each piece
of the piecewise path separately, which means that to obtain the
correct Viterbi path (corresponding to the Viterbi process) one has
to first identify the barriers in the observation sequence. In
practice, this type of tie-breaking mechanism would  complicate
implementation of the Viterbi path estimation and add significantly
to its computational cost. The solution to this problem is to allow
only strong barriers, in which case, as we saw above, the piecewise
tie-breaking can be replaced with lexicographic or co-lexicographic
tie-breaking. Fortunately, as we will see
later, the requirement of strong barriers as opposed to simply
barriers does not seem to be restrictive.

Proving the existence of the Viterbi process is the main motivation for barrier set
construction. Once it is established that the Viterbi process $V=\{V_k\}_{k \geq 1}$ exists, the next step is to study its probabilistic
 properties. An important and very useful property
 is that the process $(Z,V)=\{(Z_k, V_k)\}_{k \geq 1}$ is
 regenerative. In case of HMM, this is achieved in \cite{AVT4,Vrisk} by, roughly speaking, constructing
 regeneration times for $Z$ which are also nodes.  When it is ensured that $(Z,V)$ is regenerative, the standard theory for regenerative processes can be applied.
 See \cite{AVT4, AVacta, Vrisk,iowa} for regeneration-based
 inferences of HMM. It is important to stress that the regenerativity of $(Z,V)$ is
 possible due to the existence of barriers  and  that is an
 extra motivation of barrier construction studied in this paper. Note that the Viterbi process in Example 1.2 is not regenerative. Regenerativity of $(Z,V)$ in case of general PMM's is subject to
authors' continuing investigation; in the present paper we only deal
with the existence of $V$.

\paragraph{History of the problem.} To our best knowledge, so
far the existence of Viterbi process has been proven in the case of
HMM's only. The first attempts in that directions have made by A.
Caliebe and U. R\"osler in \cite{caliebe1,caliebe2}. They
essentially define the concept of nodes and prove the existence of
infinitely many nodes under rather restrictive assumptions like
(\ref{1-barHMM}). For an overview of the main results in
\cite{caliebe1,caliebe2}  as well as for the discussion about their
assumption, see \cite{AVT4,AVT5}. For HMM, the most general
conditions for the existence of infinitely many barriers were given
in Lemma 3.1 of \cite{AVT5} (the same lemma is also Lemma 3.1 in
\cite{AVT4}). Let us now state that lemma.

Recall that in the case of HMM $f_i$ are the emission densities with respect to measure $\mu$. Denote
\begin{align} \label{Gi}
G_i=\{x \in \X \: | \: f_i(x)>0\}, \quad i \in \Y.
\end{align}
A subset $C\subset\Y$ is called a \emph{cluster}, if
\begin{equation} \label{HMM-clustermu}
\mu\left(\cap _{i\in C}G_i\right )>0\quad {\rm and}\quad
\mu \left[ \left (\cap _{i\in C}G_i\right) \cap \left (\cup _{i\notin C}G_i \right) \right]=0,
\end{equation}
Distinct clusters need not be
disjoint and  a cluster can consist of a single state. In this
latter case such a state is not hidden, since it is indicated by any
observation it emits. When the number of states is two, then $\Y$ is
the only cluster possible, since otherwise all observations would
reveal their states and the underlying Markov chain would cease to
be hidden.

\begin{theorem} \label{HMMTh} {\rm (Lemma 3.1 in \cite{AVT5})} Suppose $Z$ is stationary HMM satisfying the following conditions.
\begin{enumerate}[label=(\roman*)]
\item \label{A1HMM} For each state $j\in \Y$
\begin{align}\label{lll}
&\mu\left(\left\{x\in\mathcal{X} \: | \: f_j(x)p_{ \Cdot j}> \max_{i\in \Y,~i\ne
j}f_i(x)p_{\Cdot i}\right\}\right)>0,\quad\text{where}~p_{\Cdot j}\stackrel{\scriptsize{\emph{def}}}{=}
\max_{i\in \Y}p_{ij}.
\end{align}
\item \label{A2HMM} There exists a
cluster $C\subset\mathcal{Y}$ such that the sub-stochastic matrix
$\mathbb{P}_C=(p_{ij})_{i,j\in C}$ is {\em primitive}, that is $\mathbb{P}^R_C$ has only positive elements for some positive
integer $R$.
\end{enumerate}
Also let Markov chain $Y$ be irreducible and aperiodic. Then for $i \in \Y$ there exists a barrier set $\X^* \subset\X^M$, $M \geq 1$, consisting of
$i$-barriers of fixed order and satisfying $P(X_{1:M} \in \X^*)>0$.
\end{theorem}

Since stationary HMM with irreducible and aperiodic $Y$ is
ergodic, it immediately follows from this theorem that under the
specified conditions almost every realization has infinitely many
barriers and piecewise construction of Viterbi process is possible.

The assumptions \ref{A1HMM} and \ref{A2HMM}  are discussed in
details in \cite{AVT4,AVT5}. Let us just mention that they are both
natural and hold in the most models in practice. In particular,
\ref{A2HMM}  is much weaker than the common assumption of having all
entries in transition matrix $(p_{ij})$ positive. It
turns out that under \ref{A2HMM}, it is possible to generalize the
existing results of exponential forgetting properties of smoothing
probabilities for HMM' \cite{Vsmoothing}. This property has nothing
to do with Viterbi paths so that \ref{A2HMM} is in a sense a natural
and desirable property from many different aspects. For another
application of \ref{A2HMM}, see \cite{peep}. The generalization of
\ref{A2HMM} in the case of PMM's is the condition \ref{Bcluster} in
Theorem \ref{thLSC} and, just like \ref{A2HMM}, also \ref{Bcluster} might
be useful for proving many other properties of PMM besides the
existence of infinite Viterbi path. It also turns out that in the
special case of 2-state HMM, both assumptions can be relaxed: namely
an irreducible aperiodic 2-state Markov chain has always primitive
transition matrix (but not necessarily having all entries positive
as it is incorrectly stated in \cite{K2}), and as argued above, the
cluster assumption \ref{A2HMM} trivially holds. It has been shown in
\cite{K2}, that for 2-state stationary HMM, almost every realization
has infinitely many barriers if
\begin{align} \label{2uneq}
\mu \left( \{ x \in \X \: | \: f_1(x) \neq f_2(x) \} \right)>0.
\end{align}
Obviously, the assumption  (\ref{2uneq}) is most natural for any HMM, so
essentially the result says that Viterbi process exists for any two-state stationary HMM.

Theorem \ref{HMMTh} does have one weakness: it does not guarantee that the barrier set $\X^*$ consists of \textit{strong} barriers. As we saw earlier, having infinitely many
\textit{strong} barriers (as opposed to simply barriers) is a very desirable property. We rectify this issue in section \ref{HMM}, where we prove a generalized version of Theorem \ref{HMMTh},
which guarantees that the barrier set $\X^*$ consists of strong barriers.

Finally we would like to add a few words on the stationarity assumption of Theorem \ref{HMMTh}. In case of HMM, this assumption is not very restrictive, since the stationary distribution of $Z$ can easily
be expressed trough the stationary distribution of $Y$. More specifically, if $(\pi_i)$ is the stationary distribution of $Y$, then
the stationary density of $Z$ is given by $p(x_1,y_1) =\pi_{y_1} f_{y_1}(x_1)$.
 For stationary HMM thus the stationary density can be easily calculated and hence the Viterbi algorithm is easy to implement. However, in general case
 of PMM we often do not have a way to calculate the stationary density
 (if it exists), and so the stationarity assumption becomes more restrictive. Therefore in the present paper we abandon this assumption altogether. This means
 that we can no longer rely on ergodicity of $Z$ to ensure that
 $P(X \in \X^* \mbox{ i.o.})=1$ for the barrier set $\X^*$ -- instead we apply the theory of \textit{Harris recurrent} Markov chains.

When the hidden state space $\Y$ is infinite, then the infinite Viterbi path of $x_{1:\infty}$ is defined trough convergences (\ref{koond}). The Viterbi process
is then defined analogously to the case when $\Y$ is finite: it is the process on $\Y$ which is almost surely the infinite Viterbi path of $X$. In \cite{Ritov}
P. Chigansky and Y. Ritov study the existence of such process in case of HMM with continuous hidden state space. The authors provide examples where the infinite Viterbi path does indeed exist,
and moreover prove its existence under certain strong log-concavity conditions for transition and observation densities. They also provide an example where $\mathcal{Y}$ is countable and Markov chain
$Y$ is positive recurrent, but the infinite Viterbi path does not exist because it diverges coordinate-wise to infinity. This is similar to our Example \ref{noInf} in the sense
that both examples demonstrate a situation where the Viterbi process does not exist.

\section{Barrier set construction theorem}\label{sec2}
Recall the definition of $p_{ij}(\cdot)$ in (\ref{pij}). We already saw that the inequalities (\ref{tr-barrier}) ensure that $(x_{t-1},x_t,x_{t+1})$ is a barrier of order 1. We generalize this idea with:
\begin{prop} \label{barPr} Suppose $b_{1:M} \in \X^M$, $M \geq 3$, is such that for some  $l \in \{2, \ldots ,M-1 \}$
\begin{align}
&p_{i1}(b_{1:l}) p_{1j}(b_{l:M}) \geq p_{ik}(b_{1:l})p_{kj}(b_{l:M}), \quad \forall i,j,k \in \Y \label{bareq1}.
\end{align}
Then $b_{1:M}$ is a 1-barrier of order $M-l$. If inequalities (\ref{bareq1}) are strict for any $i$, $j$ and any $k \neq 1$ for which the left side of the inequality is non-zero, then $b_{1:M}$ is a
strong 1-barrier of order $M-l$.
\end{prop}
\begin{proof} Let $x_{1:\infty}$ be such a realization of $X$ that $x_{t+1:t+M}=b_{1:M}$ for some $t$. Then for all $j, k \in \Y$
\begin{align}
\delta_{t+l}(1)p_{1j}(x_{t+l:t+M})&=\max_{i \in \Y} \delta_{t+1}(i)p_{i1}(x_{t+1:t+l})p_{1j}(x_{t+l:t+M}) \notag \\
& \geq \max_{i \in \Y} \delta_{t+1}(i)p_{ik}(x_{t+1:t+l})p_{kj}(x_{t+l:t+M}) \label{bareq2}\\
&=\delta_{t+l}(k)p_{kj}(x_{t+l:t+M}), \notag
\end{align}
which shows that $x_{t+1:t+M}=b_{1:M}$ is indeed a 1-node of order $M-l$. Let now inequalities (\ref{bareq1}) be strict for any $i$, $j$ and any $k \neq 1$ for which the left side of the inequality is non-zero.
Then we have for every $j \in \Y$ for which $\delta_{t+l}(1)p_{1j}(x_{t+l:t+M})>0$ and for every $k \neq 1 $ that the inequality (\ref{bareq2}) is strict, which makes $b_{1:M}$ a strong 1-barrier of order $M-l$.
\end{proof}
Proposition \ref{barPr} allows us to derive conditions \textbf{A1}-\textbf{A3} detailed below, which ensure the existence of a barrier set. For any $n \geq 2$, define
\begin{align} \label{Ypairs}
\mathcal{Y}^+(x)= \{ (i,j) \: | \: p_{ij}(x)>0\}, \quad x \in \X^n.
\end{align}
For any set $A$ consisting of vectors of length $n$ we adopt the following notation:
\begin{align*}
&A_{(k)}=\{x_{k} \:| \: x_{1:n} \in A\}, \quad 1 \leq k \leq n,\\
&A_{(k,l)}=\{x_{k:l} \:| \: x_{1:n} \in A\}, \quad 1 \leq k \leq l \leq n.
\end{align*}
Hence
$$\mathcal{Y}^+(x)_{(1)}=\{i \: | \: \exists j(i) \text{   such that }
p_{ij}(x)>0\},\quad \mathcal{Y}^+(x)_{(2)}=\{j \: | \: \exists i(j)
\text{
  such that } p_{ij}(x)>0\}.$$ Observe that if $i\in
\mathcal{Y}^+(x)_{(1)}$ and $j\in \mathcal{Y}^+(x)_{(2)}$, then not
necessarily $(i,j)\in \mathcal{Y}^+(x).$
The aforementioned conditions are the following.
\begin{description}
\item[A1]  There exists $N \geq 2$, $ n_1 < \cdots < n_{2N+2}$, set $\mathcal{X}^* \subset\mathcal{X}^{n_{2N+2}}$ and $\epsilon>0$ such for all $k=1,\ldots,2N$ and all $x \in \mathcal{X}^*_{(n_{k},n_{k+1})}$
\begin{align}
&p_{11}(x) \geq p_{i1}(x),\quad \forall i \in \mathcal{Y}, \label{ineqA11}\\
&p_{11}(x) \geq p_{1i}(x) ,\quad \forall i \in \mathcal{Y},\label{ineqA12}\\
& p_{11}(x)(1-\epsilon) > p_{ij}(x), \quad \forall i,j \in \mathcal{Y} \setminus \{1\}. \notag
\end{align}

\item[A2] There exist constants $0<\delta \leq \Delta < \infty$ such that
\begin{align*}
&p_{ij}(x) \leq \Delta, \quad \forall i,j \in \mathcal{Y}, \quad \forall  x \in \mathcal{X}^*_{(1,n_1)} \cup \mathcal{X}^*_{(n_{2N+1},n_{2N+2})},\\
&\mathcal{Y}^+(x)\neq \emptyset, \quad p_{i1}(x)  \geq \delta, \quad \forall i \in \mathcal{Y}^+(x)_{(1)}, \quad \forall  x \in \mathcal{X}^*_{(1,n_1)},\\
&\mathcal{Y}^+(x)\neq \emptyset, \quad p_{1j}(x) \geq \delta , \quad \forall j \in \mathcal{Y}^+(x)_{(2)}, \quad \forall x \in \mathcal{X}^*_{(n_{2N+1},n_{2N+2})}.
\end{align*}
\item[A3]   It holds
\begin{align*}
 \dfrac{\Delta}{\delta}(1-\epsilon)^{N}<1.
\end{align*}
\end{description}

We also consider a strengthened version of \textbf{A1}:

\begin{description}
\item[A1'] The condition \textbf{A1} holds with either inequalities (\ref{ineqA11}) or inequalities (\ref{ineqA12}) being strict for all $i \neq 1$.
\end{description}

\begin{theorem} \label{th1}
Suppose \textbf{A1}-\textbf{A3} are fulfilled. Then $\mathcal{X}^*$
consists of 1-barriers of order $n_{2N+2}-n_{N+1}$. Furthermore, if
\textbf{A1'} holds instead of \textbf{A1}, then the 1-barriers are
strong.
\end{theorem}

Note how the condition \textbf{A1} concerns only the section $\X^*_{(n_1,n_{2N+1})}$ of $\X^*$ while the condition \textbf{A2} concerns the sections $\X^*_{(1,n_1)}$
and $\X^*_{(n_{2N+1},n_{2N+2})}$. This motivates:

\begin{definition} If $\X^*$ satisfies \textbf{A1} (\textbf{A1}'), then $\X^*_{(n_1,n_{2N+1})}$ is called \emph{a (strong) center part of a barrier set}.
\end{definition}

Since a center part of a barrier set has a cyclic structure (consisting of $2N$ cycles), then it is natural that its construction is also cyclical. Consider for example the case when $\X$ is
 discrete and there exists a sequence $x_{1:n} \in \X^n$, $n \geq 2$, such that
\begin{align*}
x_1=x_n \quad \mbox{and} \quad p_{11}(x_{1:n})>p_{ij}(x_{1:n}), \quad \forall (i,j)\in \Y^2 \setminus \{(1,1)\}.
\end{align*}
Thus, denoting $x=x_{1:n-1}$, for any $N \geq 2$ we can take the strong center part to be
\begin{align*}
\{(\underbrace{x, x, \ldots ,x}_{\text{$2N$ blocks of $x$}},x_n)\}.
\end{align*}
Since we can take $N$ arbitrarily large, we can always ensure that \textbf{A3} holds when $\delta$ and $\Delta$ are fixed.

Let us consider now the case when $\X$ is uncountable. Then the situation is in general more complicated, since center part of a barrier set should typically contain uncountably many vectors for $X$
to return to the corresponding barrier set $\X^*$ infinitely often. For any vector sets $A \subset\X^k$, $k \geq 1$, and $B \subset\X^l$, $l \geq 1$, we write
\begin{align} \label{dotOp}
A \cdot B = \{x_{1:k+l-1} \:| \: x_{1:k} \in A, \: x_{k:k+l-1} \in B\}.
\end{align}
For a fixed $\epsilon>0$ and $n \geq 1$ let
\begin{align*}
W=\{x_{1:n} \in \X^n \: | \: p_{11}(x_{1:n})(1-\epsilon)>p_{ij}(x_{1:n}), \: (i,j) \in \Y^2 \setminus \{(1,1)\} \}.
\end{align*}
We can construct a strong center part of a barrier set by gluing
together $2N$ instances of $W$:
\begin{align} \label{Wset}
\underbrace{W \cdots W}_{2N \text{ instances of $W$}}.
\end{align}
Again, for a fixed $\delta$ and $\Delta$, here $N$ can be taken so
large that \textbf{A3} holds. However, for $X$ to enter the
corresponding barrier set $\X^*$ infinitely often, the set
(\ref{Wset}) must have positive $\mu^{2N\cdot (n-1)+1}$ measure.
This might be difficult to confirm for specific models. In fact,
depending on $\epsilon$, $n$ and $N$, the set (\ref{Wset}) might
well be empty. But in many instances (\ref{Wset}) does have a
positive $\mu^{2N\cdot (n-1)+1}$-measure, regardless of the choice
of $N$. The following example demonstrates this.
\begin{example}\label{exSt} Let the function $(x,x') \mapsto q(x,i|x',j)$ be continuous for all $i,j \in \Y$. Suppose there exists $x_{1:n} \in \X^n$, $n\geq 2$, such that
\begin{align*}
x_1=x_n \quad \mbox{and} \quad  p_{11}(x_{1:n}) > p_{ij}(x_{1:n}), \quad \forall (i,j) \in \Y^2 \setminus \{(1,1)\}.
\end{align*}
Since $(x,x') \mapsto q(x,i|x',j)$ are continuous, then so must be maps $\X^n \ni x \mapsto p_{ij}(x)$, and therefore there must exist open balls $B_1,\ldots,B_{n-1} \subset\X$ and $\epsilon>0$ such that $x_{1:n} \in B_1 \times \cdots \times B_{n-1} \times B_{1}$ and for every $x \in B_1 \times \cdots \times B_{n-1} \times B_{1}$
\begin{align*}
& p_{11}(x)(1-\epsilon) > p_{ij}(x), \quad \forall (i,j) \in \Y^2 \setminus \{(1,1)\}.
\end{align*}
Setting $B=B_1 \times \cdots \times B_{n-1}$, we have for arbitrary $N \geq 2$ that set
\begin{align*}
\underbrace{B \times \cdots \times B}_{2N \text{ blocks of $B$}}\times B_1
\end{align*}
is a strong center part of a barrier set. Assuming that any open ball has positive $\mu$-measure, this barrier set must have positive $\mu^{2N\cdot (n-1)+1}$-measure.
\end{example}

\begin{proof}[Proof of Theorem \ref{th1}]
Fix $x_{1:n_{2N+2}} \in \X^*$. We will show that \textbf{A1}-\textbf{A3} imply inequalities
\begin{align}
&p_{i1}(x_{1:n_{N+1}}) \geq  p_{ij}(x_{1:n_{N+1}}), \quad \forall i,j \in \Y, \label{pi1Ineq} \\
&p_{1i}(x_{n_{N+1}:n_{2N+2}}) \geq  p_{ji}(x_{n_{N+1}:n_{2N+2}}), \quad \forall i,j \in \Y. \label{p1iIneq}
\end{align}
We also show that if \textbf{A1'} holds instead of \textbf{A1}, then
either inequalities (\ref{pi1Ineq}) or (\ref{p1iIneq}) are strict
for all $i$ and $j \neq 1$ for which the left side of the inequality
is non-zero. Then the statement follows from Proposition
\ref{barPr}. Denote
\begin{align*}
a_{ij}(0)=p_{ij}(x_{1:n_1}), \quad a_{ij}(k)=p_{ij}(x_{n_k:n_{k+1}}), \quad k=1,\ldots,N.
\end{align*}
We start by proving (\ref{pi1Ineq}). If $i \notin \Y^+(x_{1:n_1})_{(1)}$, then (\ref{pi1Ineq}) holds, because for every $j \in \Y$
\begin{align*}
p_{ij}(x_{1:n_{N+1}})=\max_{y \in \Y}p_{iy}(x_{1:n_1})p_{yj}(x_{n_1:n_{N+1}})=\max_{y \in \Y}0 \cdot p_{yj}(x_{n_1:n_{N+1}})=0.
\end{align*}
Consider now the case where  $i \in \Y^+(x_{1:n_1})_{(1)}$. Let $M=M(i)$ be the set of all vectors $y_{1:N+1}$ which maximise the expression
\begin{align} \label{maxLik}
a_{iy_1}(0) \prod_{k=1}^{N}a_{y_{k}y_{k+1}}(k).
\end{align}
Hence for every $y_{1:2N+1} \in M$, (\ref{maxLik}) is equal to $\max_{j \in \Y}p_{ij}(x_{1:n_{N+1}})$. First we will prove that for any $y_{1:N+1} \in M$, $y_{1:N}$ contains at least one 1, i.e.
\begin{align} \label{contains1}
M \cap (\Y \setminus \{1\}) \times \Y = \emptyset.
\end{align}
Assuming on contrary, we would have
\begin{align*}
\max_{j \in \Y}p_{ij}(x_{1:n_{N+1}})&= \max_{y_{1},..,y_{N} \in \mathcal{Y}\setminus \{1\}, \: y_{N+1} \in \Y}   a_{iy_1}(0) \prod_{k=1}^{N}a_{y_{k}y_{k+1}}(k) \\
& \stackrel{\mbox{\footnotesize{\textbf{A1}, \textbf{A2} }}}{\leq} \dfrac{a_{i1}(0)}{\delta} \cdot \Delta \cdot (1-\epsilon)^{N}  \cdot \prod_{k=1}^{N}a_{11}(k)  \\
&\stackrel{\mbox{\footnotesize{\textbf{A3}}}}{<}a_{i1}(0) \prod_{k=1}^{N}a_{11}(k) \\
& \leq  \max_{j \in \Y}p_{ij}(x_{1:n_{N+1}})
\end{align*}
- a contradiction. Fix $y_{1:N+1}' \in M$ arbitrarily; as we saw, there must exist $u \in \{1, \ldots ,N\}$ such that $y'_u=1$. Since
\begin{align*}
\max_{j \in \Y}p_{ij}(x_{1:n_{N+1}})&=a_{iy_1'}(0) \prod_{k=1}^{N}a_{y'_{k}y'_{k+1}}(k)\\
& \leq a_{iy_1'}(0)a_{y_1'y_2'}(1)\cdots a_{y_{u-1}'1}(u-1) \cdot a_{11}(u) a_{11}(u+1) \cdots a_{11}(N)\\
& \leq p_{i1}(x_{1:n_{N+1}}),
\end{align*}
then (\ref{pi1Ineq}) holds. The proof of inequalities (\ref{p1iIneq}) is symmetrical.

Finally, we need to show that if \textbf{A1'} holds, then either inequalities (\ref{pi1Ineq}) or (\ref{p1iIneq}) are strict for all $i$ and $j \neq 1$ for which the left side of the inequality is non-zero. For this it suffices to prove the following two claims:
\begin{enumerate}[label=(\roman*)]
\item if inequalities (\ref{ineqA12}) are strict for all $i \neq 1$, then inequalities (\ref{pi1Ineq}) are strict for all $i \in \Y^+(x_{1:n_1})_{(1)}$ and $j \neq 1$;
\item if inequalities (\ref{ineqA11}) are strict for all $i \neq 1$, then inequalities (\ref{p1iIneq}) are strict for all $i \in \Y^+(x_{n_{2N+1}:n_{2N+2}})_{(2)}$ and $j \neq 1$.
\end{enumerate}
We only prove the first claim; the proof for the second claim is symmetrical. Let the inequalities (\ref{ineqA12}) be strict for all $i \neq 1$. Let now again $i \in \Y^+(x_{1:n_1})_{(1)}$ and let $j \neq 1$. Note that by \textbf{A1} and \textbf{A2} $p_{i1}(x_{1:n_{N+1}})>0$.  We show now that assumption
\begin{align} \label{contr}
 p_{i1}(x_{1:n_1})=p_{ij}(x_{1:n_1})
\end{align}
leads to contradiction. Indeed, assuming (\ref{contr}), we have that there exists sequence $y'_{1:N+1}$ which belongs  to set $M(i)$ and for which $y'_{N+1}=j\neq 1$. Therefore by (\ref{contains1}) there must exist $u \in \{1, \ldots ,N\}$ such that $y'_u=1$. Then
\begin{align*}
p_{ij}(x_{1:n_{N+1}})&=a_{iy_1'}(0) \prod_{k=1}^{N}a_{y'_{k}y'_{k+1}}(k)\\
& < a_{iy_1'}(0)a_{y_1'y_2'}(1)\cdots a_{y_{u-1}'1}(u-1) \cdot a_{11}(u) a_{11}(u+1) \cdots a_{11}(N)\\
& \leq p_{i1}(x_{1:n_{N+1}}).
\end{align*}
\end{proof}

Theorem \ref{th1} gives conditions for constructing the barrier set $\X^*$, but we also need to ensure that $X$ enters into $\X^*$ infinitely often a.s. For this we will use the following
\begin{prop}\label{prop1} Let $\X^* \subset\X^M$ for some $M \geq 1$. If for some $A \subset \Z$ and $\epsilon>0$ it holds
\begin{align*}
&P(Z_k \in A \mbox{ i.o.})=1,\\
&P(X_{1:M} \in \mathcal{X}^*|Z_1=z)\geq \epsilon, \quad \forall z \in A,
\end{align*}
then $P(X \in \X^* \mbox{ i.o.})=1$.
\end{prop}
\begin{proof} Take $B=\{((x_1,y_1), \ldots, (x_M,y_M)) \: | \: x_{1:M} \in \X^*, y_{1:M} \in \Y^M\}$. Thus
\begin{align*}
&P(Z_{1:M} \in B|Z_1=z)=P(X_{1:M} \in \mathcal{X}^*|Z_1=z)\geq \epsilon, \quad \forall z \in A.
\end{align*}
From Lemma \ref{lem1} it follows that $P(Z \in B \mbox{ i.o.})=1$ which implies $P(X \in \X^* \mbox{ i.o.})=1$. \end{proof}
\paragraph{Harris chains and reachable points.} We will now introduce some general state space Markov chain
terminology. Markov chain $Z$ is called
\textit{$\varphi$-irreducible} for some $\sigma$-finite measure
$\varphi$ on $\B(\Z)$, if $\varphi(A)>0$ implies $\sum_{k=2}^\infty
P(Z_k \in A|Z_1=z)>0$ for all $z \in \Z$. If $Z$ is
$\varphi$-irreducible, then there exists (see \cite[Prop.
4.2.2.]{MT}) a \textit{maximal irreducibility measure} $\psi$ in the
sense that for any other irreducibility measure $\varphi'$ the
measure $\psi$ dominates $\varphi'$, $\psi \succ \varphi'$. The
symbol $\psi$ will be reserved to denote the maximal irreducibility
measure of $Z$.  A point $z\in \Z$ is called \textit{reachable} if for
every open neighbourhood $O$ of $z$,
\begin{align*}
\sum_{k =2}^\infty P(Z_k \in O|Z_1=z')>0, \quad \forall z' \in \Z.
\end{align*}
For $\psi$-irreducible $Z$, the point $z$ is reachable if and only
if it belongs to the support of $\psi$ \cite[Lemma 6.1.4]{MT}. Since we have equipped space $\Z$ with product topology $\tau \times 2^\Y$, where $\tau$ denotes the topology induced by the metrics of
$\X$, the above-stated definition of reachable point is actually
equivalent to the following: point $(x,i) \in \Z$ is called
\emph{reachable}, if for every open neighbourhood $O$ of $x$,
\begin{align*}
\sum_{k =2}^\infty P(Z_k \in O \times \{i\}|Z_1=z)>0, \quad \forall
z \in \Z.
\end{align*}
Chain $Z$ is called \textit{Harris recurrent}, if it is
$\psi$-irreducible and $\psi(A)>0$ implies $P(Z_k \in A \mbox{
i.o.}|Z_1=z)=1$ for all $z \in \Z$.

The following lemma links  the conditions of
Proposition \ref{prop1} to the conditions \textbf{A1}-\textbf{A2}
and Harris recurrence of $Z$.
\begin{lemma}\label{lemio} Let $\X^* \subset\X^M$ satisfy \textbf{A1} and \textbf{A2} and let $Z$ be Harris recurrent.
Moreover, assume that there exists $i \in \Y$ such that $i \in \Y^+(x)_{(1)}$ for every $x \in \X^*_{(1,n_1)}$. 
Denote
\begin{align*}
\X^*(x_1)=\{x_{2:M} \:| \: x_{1:M} \in \X^*\}, \quad x_1 \in \X^*_{(1)}.
\end{align*}
 If 
\begin{align} \label{muxPos}
 \mu^{M-1}(\X^*(x_1))>0, \quad \forall x_1 \in \X^*_{(1)},
\end{align} 
and $\psi(\X^*_{(1)} \times \{i\})>0$, then $P(X \in \X^* \mbox{ i.o.})=1$.
\end{lemma}
\begin{proof}
By \textbf{A1}, \textbf{A2} and \eqref{muxPos} we have for every $x_1 \in \X^*_{(1)}$
\begin{align*}
&P(X_{1:n_{2N+2}} \in \X^*|Z_1=(x_1,i))\\
 & = \int_{\X^*(x_1)} \sum_{y_{1:M} \colon y_1=i}p(x_{2:M},y_{2:M}|x_1,y_1) \, \mu^{M-1}(dx_{2:M})\\
 &\geq \int_{\X^*(x_1)} p_{i1}(x_{1:n_1}) \left( \prod_{k=2}^{2N+1} p_{11}(x_{n_{k-1}:n_k})\right) \max_{j \in \Y}p_{1j}(x_{n_{2N+1}:n_{2N+2}}) \, \mu^{M-1}(dx_{2:M})\\
 &> 0.
\end{align*}
Thus there must exist $A \subset\X^*_{(1)} \times \{i\}$ and $\epsilon'>0$ such that $\psi(A)>0$ and $P(X \in \X^*|Z_1=z) \geq \epsilon'$ for all $z \in A$. Since $Z$ is Harris recurrent,
then $P(Z_k \in A \mbox{ i.o.})=1$ and the statement follows from Proposition \ref{prop1}.
\end{proof}
\section{Barrier set construction with lower semi-continuous transition densities} \label{sec3}
In Subsection \ref{HMM} we will show how Theorem \ref{th1} can be used to derive simple and general conditions for the existence of
infinite Viterbi path in case of HMM. For non-HMM's the situation may be more complex and proving \textbf{A1'}, \textbf{A2} and \textbf{A3} might be difficult.
In the present section we derive some conditions which are easier to handle by assuming lower semi-continuity and boundedness of functions $(x,x') \mapsto q(x,j|x',i)$.
In what follows, the first assumption {\bf B1} is
closely related to the condition \textbf{A2} and the second
assumption {\bf B2} guarantees the existence of a strong center part
of a barrier set.
\begin{description}
 \item[B1\namedlabel{Bcluster}{\textbf{B1}}] There exists an open set $E \subset\X^q$, $q\geq 2$, such that $\Y^+ \stackrel{\scriptsize{\mbox{def}}}{=}\Y^+(x)$
 is the same for every $x \in E$ and satisfies the following property: $(i,j) \in \Y^+$ for every $i \in \Y^+_{(1)}$ and $j \in \Y^+_{(2)}$. Furthermore, we assume
 that there exists a reachable point $(x_E,i_E)$ in $E_{(1)}\times \Y^+_{(1)}$.
\item[B2\namedlabel{Bcenter}{\textbf{B2}}] For arbitrary $N \geq 2$ there exists a strong center part of a barrier set $\X^*_{(n_1,n_{2N+1})}$ which
 is open, non-empty and has $2N$ cycles. We assume that both set $\X^*_{(n_1)}$ and parameter $\epsilon$ of \textbf{A1} are independent of $N$, and
there exists a compact set $K \subset \X$, which is independent of $N$, such that $\X^*_{(n_{2N+1})}$ is contained in $K$. Furthermore, we assume that
there exists $x^* \in \X^*_{(n_1)}$ such that $(x^*,1)$ is reachable.
\end{description}
\begin{theorem} \label{thLSC} Let $\mu$ be strictly
postive\footnote{A measure is called \textit{strictly positive} if
it assigns a positive measure to all non-empty open sets.} and let
for every pair of states $i,j \in \Y$ function $(x,x') \mapsto
q(x,i|x',j)$ be lower semi-continuous and bounded. If $Z$ satisfies
\ref{Bcluster} and \ref{Bcenter}, then there exists $\X^*$ satisfying
\textbf{A1'}, \textbf{A2} and \textbf{A3}. Moreover, if $Z$ is
Harris recurrent, then $P(X \in \X^* \mbox{ i.o.})=1$.
\end{theorem}

Before  proving the  theorem, let us briefly discuss
its assumptions. Under
$\psi$-irreducibility (that is implied by Harris recurrence) the
existence of certain reachable points is not restrictive, because
any point in the support of $\psi$ is reachable. Thus  {\bf B1} is
merely to guarantee that $(i,j) \in \Y^+$ for every $i \in
\Y^+_{(1)}$ and $j \in \Y^+_{(2)}$ (we have already noted that this
need not hold in general). It turns out that for stationary $Z$ this
property is closely related to {\it  subpositivity} property of
factor maps in ergodic theory. We shall return to  that connection
and also discuss the necessity of {\bf B1} in Subsection \ref{disc}.
{\bf B2} provides some necessary assumptions for cycle-construction.
Typically the center part of a barrier set is constructed by glueing
together some fixed cycles and {\bf B2} basically guarantees that no
matter how many cycles are connected, the last one always ends in a
fixed compact set $K$. As we shall see in the examples, this
condition holds for many models.
\begin{proof}[Proof of Theorem \ref{thLSC}] \textbf{Lower likelihood bound for vectors in $E$.} We will show that with no loss of generality we may assume that there exist $\delta_0>0$
such that \begin{align} \label{delta0}
p_{ij}(x_{1:q}) > \delta_0, \quad \forall x_{1:q} \in E, \quad \forall (i,j) \in \Y^+.
\end{align}

Let $x_{1:q}' \in E$ be such that $x_1'=x_E$.  Denote $\delta_0=\dfrac{1}{2} \min_{(i,j) \in \Y^+}p_{ij}(x'_{1:q})$. Since $\Y^+(x'_{1:q})=\Y^+$, then $\delta_0>0$. Denote
\begin{align*}
E'=\left\{x_{1:q} \in E \:| \: \min_{(i,j) \in \Y^+}p_{ij}(x_{1:q})>\delta_0 \right\}
\end{align*}
It is not difficult to confirm by induction that lower semi-continuity and boundedness of $(x,x') \mapsto q(x,i|x',j)$ implies lower semi-continuity of
functions $x_{1:n} \mapsto p(x_{2:n},y_{2:n}|x_1,y_1)$ for all $n \geq 2$. Hence for all $i,j \in \Y$ the function $x_{1:q} \mapsto p_{ij}(x_{1:q})$ is lower
semi-continuous since it expresses as a maximum over lower semi-continuous functions. Therefore the function
$x_{1:q} \mapsto \min_{(i,j) \in \Y^+} p_{ij}(x_{1:q})$ must also be lower semi-continuous, and so $E'$ must be open. Also $x_E \in E'_{(1)}$. Therefore $E'$
 could play the role of $E$ and so there is no loss of generality in assuming that (\ref{delta0}) holds true. \\

\noindent \textbf{Construction of set $D_1$.} Recall the element $x_E$ from \ref{Bcluster}. Next we will show that there exist $l_1 > q$, $\delta_1>0$ and an open
set $D_1 \subset\X^{l_1}$ such that $x_E \in D_{1(1)}$, $D_{(l_1)} \subset \X^*_{(n_1)}$ and
\begin{align} \label{D1ineq1}
&\Y^+(x)=\Y^+, \quad \forall x \in D_1,\\
&p_{i1}(x) \geq \delta_1, \quad \forall i \in \Y^+_{(1)}, \quad \forall x \in D_1. \label{D1ineq2}
\end{align}

Fix $x_{1:q}' \in E$ such that $x_1'=x_E$. Also fix $j' \in \Y^+_{(2)}$. By \ref{Bcenter} set $\X^*_{(n_1)}$ is open (projection is an open map) and contains an element
$x^*$ such that $(x^*,1)$ is reachable, so there must $k \geq 1$ such that $P(Z_{k+1} \in \X^*_{(n_1)}\times \{1\}|Z_1=(x_q',j'))>0$. This implies that there exists
$x'_{q+1:q+k} \in \X^{k}$ such that $x_{q+k}' \in \X^*_{(n_1)}$ and $\epsilon_0 \DEF p_{j'1}(x'_{q:q+k})>0$. We define
\begin{align*}
D_1= \left\{x_{1:q+k} \:| \: \min_{i \in \Y^+_{(1)}}p_{i 1}(x_{1:q+k})>\delta_0 \epsilon_0, \: x_{1:q} \in E , \: x_{q+k} \in \X^*_{(n_1)} \right\} ,
\end{align*}
where $\delta_0$ is the constant from \eqref{delta0}. Set $D_1$ is open by the fact that function $x_{1:q+k} \mapsto \min_{i \in \Y^+_{(1)}} p_{i1}(x_{1:r+k})$ is lower
semi-continuous and both $E$ and $\X^*_{(n_1)}$ are open. Note that inequality \eqref{D1ineq2} is satisfied with $\delta_1= \delta_ 0\epsilon_0$. When $x_{1:q+k} \in D_1$, then
\begin{itemize}
\item $\Y^+(x_{1:q+k})_{(1)} \supset \Y^+_{(1)}$ by the fact that $\min_{i \in \Y^+_{(1)}}p_{i 1}(x_{1:q+k})>0$;\\
\item $\Y^+(x_{1:q+k})_{(1)} \subset \Y^+_{(1)}$, because by \ref{Bcluster} $\Y^+(x_{1:q})=\Y^+$.
\end{itemize}
Hence \eqref{D1ineq1} holds. By \ref{Bcluster} $(i,j') \in \Y^+$ for every $i \in \Y^+_{(1)}$. Hence by \eqref{delta0} we have for every $i \in \Y^+_{(1)}$ that
$p_{i 1}(x'_{1:q+k}) \geq p_{i j'}(x'_{1:q})p_{j'1}(x'_{q:q+k})> \delta_0 \epsilon_0$. This implies that $x'_{1:q+k} \in D_1$ and so $x_E=x_1' \in D_{1(1)}$, as required.\\

\noindent \textbf{Construction of sets $D_2(x)$.} Recall now the compact set $K$ from \ref{Bcenter}. We will show that there exists a constant $\delta_2>0$ such that the
following holds: for every $x \in K$ there exists $l_2(x) > q$ and a non-empty open set $D_2(x) \subset\X^{l_2}$ such that
\begin{align} \label{D2ineq1}
&\Y^+(x,x_{1:l_2})_{(2)}=\Y^+_{(2)}, \quad \forall x_{1:l_2} \in D_2(x),\\
&p_{1j}(x,x_{1:l_2}) \geq \delta_2, \quad \forall j \in \Y^+_{(2)}, \quad \forall x_{1:l_2} \in D_2(x). \label{D2ineq2}
\end{align}

Denote for $s \geq 2$
\begin{align*}
p_s(x)=P(Z_s \in E_{(1)} \times \{i_E\}| Z_1=(x,1))
\end{align*}
and $G_s=\{x \in \X \: | \: p_s(x)>0\}$. Functions $x \mapsto p_s(x)$ are lower semi-continuous by Fatou's Lemma and lower semi-continuity of functions $x \mapsto p(x_{2:s},y_{2:s}|x_1=x,y_1=1)$, and so the sets $G_s$ must be open. By \ref{Bcluster} $E_{(1)}$ is open (projection is an open map) and  contains an element $x_E$ such that $(x_E,i_E)$ is reachable for an $i_E \in \Y^+_{(1)}$. This implies that sets $G_s$ form an open cover of compact set $K$. Hence there exists an $s_0 \geq 2$ such that $K \subset \cup_{s =2}^{s_0} G_s$.

Define now
\begin{align*}
&h_s(x)=\sup_{x_{2:s} \in \X^{s-2}\times E_{(1)}}p(x_{2:s},y_s=i_E|z_1=(x,1)),\\
&s(x)=\argmax_{s \in \{2, \ldots, s_0\} } h_s(x)
\end{align*}
and
\begin{align*}
\epsilon(x)=\dfrac{1}{2} \max_{s \in \{2,\ldots,s_0 \} }h_s(x).
\end{align*}
Note that when $x \in K$, then $\epsilon(x)>0$. Indeed, when $x \in K$ then there exists $s \in \{2, \ldots ,s_0\}$ such that $x \in G_{s}$. Hence $p_s(x)>0$, and so $h_s(x)>0$. This implies that $\epsilon(x)>0$.

Denote
\begin{align*}
F(x)=\{x_{2:s(x)} \:| \: p(x_{2:s(x)},y_{s(x)}=i_E|z_1=(x,1))> \epsilon(x)\} \cap \X^{s(x)-2} \times E_{(1)}.
\end{align*}
Functions $x_{2:s} \mapsto p(x_{2:s},y_s=i_E|z_1=(x,1))$ must be lower semi-continuous, since they express as a finite sum of bounded lower semi-continuous functions. Therefore the sets $F(x)$ must be open. Also, when $x \in K$, then, as we saw, $\epsilon(x)>0$, and so $F(x)$ is non-empty. We define
\begin{align*}
&l_2(x)=s(x)+q-2,\\
&D_2(x)=F(x) \cdot E,
\end{align*}
where operator $\cdot$ is defined in (\ref{dotOp}). Set $D_2(x)$ must be open, as it can be expressed as an intersection of two open sets:
\begin{align*}
D_2(x)=\X^{s(x)-2}\times E \cap F(x) \times \X^{q-1}.
\end{align*}
Set $D_2(x)$ is also non-empty for all $x \in K$ by the fact that sets $F(x)$ are non-empty and by definition of sets $F(x)$ and $D_2(x)$.

Next, we prove the existence of $\delta_2>0$. Fix $x \in K$. Set $l=l_2(x)$ and $s=s(x)$ and let $x_{2:l+1} \in D_2(x)$.
By definition of $F(x)$,
\begin{align*}
\sum_{y_{2:s-1} }p(x_{2:s},y_{2:s-1},y_s=i_E|z_1=(x,1))>\epsilon(x),
\end{align*}
and so
\begin{align} \label{lBound}
p_{1i_E}(x,x_{2:s})>|\Y|^{-s}\epsilon(x)\geq |\Y|^{-s_0}\epsilon(x)>0.
\end{align}
Fix $j \in \Y^+_{(2)}$. Thus we have by \ref{Bcluster} that $(i_E,j) \in \Y^+$. Therefore by (\ref{lBound}) and (\ref{delta0})
\begin{align} \label{ineqj}
p_{1j}(x,x_{2:l+1})\geq p_{1i_E}(x,x_{2:s})p_{i_Ej}(x_{s:l+1}) \geq |\Y|^{-s_0}\epsilon(x)\delta_0>0.
\end{align}
Thus $\Y^+(x,x_{2:l+1})_{(2)} \supset \Y^+_{(2)}$; since also $\Y^+(x,x_{2:l+1})_{(2)} \subset \Y^+_{(2)}$, then \eqref{D2ineq1} must hold.
Let now $x$ vary. Since $\epsilon(x)$ is a lower semi-continuous function, then it follows from \eqref{ineqj} and compactness of $K$ that \eqref{D2ineq2} holds with $\delta_2=|\Y|^{-s_0}\min_{x \in K}\epsilon(x) \cdot \delta_0>0$.\\

\noindent \textbf{Construction of $\X^*$.}  By the boundedness assumption there exists $\Delta_0>1$ such that $q(z|z') \leq \Delta_0$ for all $z,z' \in \Z$. Denote $l_{\max}=l_1 \vee (s_0+q-2)$, where $\vee$ denotes maximum. Hence $l_1 \vee \max_{x \in K}l_2(x) \leq l_{\max}$. We take $\Delta= \Delta_0^{l_{\max}}$ and $\delta= \delta_1 \wedge \delta_2$, where $\wedge$ denotes minimum. Take now $N \geq 2$ so large that \textbf{A3} holds -- this is possible because according to \ref{Bcenter} set $\X^*_{(n_1)}$, $\epsilon$ and $K$ are all independent of $N$.

We note that the set $D_1 \cdot \X^*_{(n_1,n_{2N+1})}$ is open and non-empty, since $D_1$ is open and non-empty by construction, $\X^*_{(n_1,n_{2N+1})}$ is open and non-empty by \ref{Bcenter} and $D_{1(l_1)} \subset \X^*_{(n_1)} \neq \emptyset$ by construction of $D_1$. Hence, taking $n_1=l_1$, there exist open balls $( B_k )_{k=1}^{n_{2N+1}}$ in $\X$ such that $x_E \in B_1$ and
\begin{align} \label{inc1}
B_1 \times \cdots \times B_{n_{2N+1}} \subset D_1 \cdot \X^*_{(n_1,n_{2N+1})}.
\end{align}
Denote $\X(l)=\{x \in \X \:|\: l_2(x)=l\}$. The sets $(\X(l))_{l=1}^{l_{\max}}$ form a finite cover of $\X$. Therefore by the assumption that measure $\mu$ is strictly positive there must exist positive integer $l_2 \leq l_{\max}$ such that denoting $\X_0=B_{n_{2N+1}} \cap \X(l_2)$, we have
\begin{align} \label{X0pos}
\mu(\X_0)>0.
\end{align}
Take $D_2= \cup_{x \in \X_0} \{x\} \times D_2(x)$ and
\begin{align*}
\X^*= B_{1} \times \cdots \times B_{n_{2N+1}-1} \times D_2.
\end{align*}
Then \textbf{A1}' is satisfied with $n_1=l_1$ and $n_{2N+2}-n_{2N+1}=l_2$ by the fact that $\X^*_{(n_1,n_{2N+1})}$ is a strong center part of a barrier set (\ref{Bcenter}) and by \eqref{inc1}. \textbf{A2} is satisfied by \eqref{D1ineq1}, \eqref{D1ineq2} \eqref{D2ineq1}, \eqref{D2ineq2} and \eqref{inc1}.

Let now $Z$ be Harris recurrent. To complete the proof it suffices to show that $P(X \in \X^* \mbox{ i.o.})=1$. To prove this, we show that the assumptions of Lemma \ref{lemio} are fulfilled. We start with proving the assumption
\begin{align} \label{muPos}
\mu^{M-1}(\X^*(x_1))>0, \quad \forall x_1 \in \X^*_{(1)},
\end{align}
where we define $M = n_{2N+2}$ and
\begin{align*}
\X^*(x_1) = \{x_{2:M} \:| \: x_{1:M} \in \X^*\}, \quad x_1 \in \X^*_{(1)}.
\end{align*}
First we note that set $D_2$ is measurable. Indeed, setting $s=l_2-q+2$, set $D_2$ expresses as
\begin{align*}
\left( \cup_{x_1 \in \X_0}\{x_1\} \times F(x_1) \right) \cdot E=\left( \{x_{1:s} \:| \: p(x_{2:s},y_s=i_E|x_1,y_1=1)> \epsilon(x_1)\} \cap \X_0 \times \X^{s-1} \right) \cdot E.
\end{align*}
The function $x_{1:s} \mapsto p(x_{2:s},y_s=i_E|x_1,y_1=1)-\epsilon(x_1)$ is
measurable, so $D_2$ must be measurable. Next, note that $\mu^{l_2}(D_2(x))>0$ for all $x \in \X_0 \subset K$ by the fact that sets $D_2(x)$ are by construction open and non-empty. Together with
\eqref{X0pos} the observations above imply that
\begin{align*}
\mu^{l_2+1}(D_2)= \int_{\X_0} \int_{D_2(x_1)} \, \mu^{l_2}(d x_{2:l_2+1}) \, \mu(dx_1)>0
\end{align*}
which in turn implies \eqref{muPos}.

Since $Z$ is Harris recurrent, then it is by definition
$\psi$-irreducible. To prove the rest of the assumptions of Lemma
\ref{lemio}, it suffices to show that
\begin{align} \label{iinc}
i_E \in \Y^+(x)_{(1)}, \quad \forall x \in \X^*_{(1,n_1)}
\end{align}
and
\begin{align} \label{psiPos}
\psi(\X^*_{(1)} \times \{i_E\})>0.
\end{align}
By \eqref{D1ineq1} $\Y^+_{(1)}(x)=\Y^+_{(1)}$ for every $x \in \X^*_{(1,n_1)}=\X^*_{(1,l_1)}$; also by \ref{Bcluster} $i_E \in \Y^+_{(1)}$, so \eqref{iinc} holds. Since point $(x_E,i_E)$ is reachable by \ref{Bcluster}, then this point belongs to the support of measure $\psi$. Since $\X^*_{(1)} \times \{i_E\}$ is an open neighbourhood of $(x_E,i_E)$ (recall that $x_E \in \X^*_{(1)}=B_1$), then \eqref{psiPos} holds by definition of measure support.
\end{proof}

\section{Examples}\label{sec:examples}
\subsection{Hidden Markov model} \label{HMM}
For HMM, Theorem \ref{th1} allows us to deduce a generalized version of Theorem \ref{HMMTh}. Recall the definitions of $G_i$ (\ref{Gi}). We introduce a new term obtained by weakening the cluster condition (\ref{HMM-clustermu}): a subset $C\subset\Y$ is called a \emph{weak cluster}, if
\begin{equation*}
\mu \left[ \left(\cap _{i\in C}G_i \right)\setminus \left(\cup _{i\notin C}G_i \right) \right]>0.
\end{equation*}
The result for HMM is the following:
\begin{cor} \label{HMMcor} Suppose $Z$ is HMM satisfying the following conditions.
\begin{enumerate}[label=(\roman*)]
\item \label{HMMcon1} For each state $j\in \Y$
\begin{align*}
&\mu\left(\left\{x\in\mathcal{X} \: | \: f_j(x)p_{ \Cdot j}> \max_{i\in \Y,~i\ne
j}f_i(x)p_{\Cdot i}\right\}\right)>0,\quad\text{where}~p_{\Cdot j} \stackrel{\mbox{\emph{\scriptsize{def}}}}{=}
\max_{i\in \Y}p_{ij}.
\end{align*}
\item \label{HMMcon2} There exists a \emph{weak}
cluster $C\subset\mathcal{Y}$ such that the sub-stochastic matrix
$\mathbb{P}_C=(p_{ij})_{i,j\in C}$ is primitive in the sense
that $\mathbb{P}^R_C$ consists of only positive elements for some positive
integer $R$.
\end{enumerate}
Also let Markov chain $Y$ be irreducible. Then there exist $i \in \Y$ and a barrier set $\X^*$ consisting of strong $i$-barriers of fixed order and satisfying $P(X \in \X^* \mbox{ i.o.})=1$.
\end{cor}
Compared to Theorem \ref{HMMTh} we have removed the assumption of stationarity of $Z$ and aperiodicity of $Y$, and replaced the assumption that $C$ is cluster with a substantially weaker assumption that it is a weak cluster. Also, the result above guarantees the existence of infinitely many \textit{strong} nodes, instead of just nodes like in Theorem \ref{HMMTh}.
\begin{proof}[Proof of Corollary \ref{HMMcor}]
Fix $j_1 \in \Y$. Denote for every $k\geq 2$
\begin{align*}
j_{k}=\argmax_{j \in \Y}p_{jj_{k-1}}.
\end{align*}
There must exist integers $u$ and $v$, $u < v$, such that $j_u=j_v$. We denote $n=v-u+1$ and $i_{1:n}=(j_v,j_{v-1},\ldots, j_u)$. If needed, we will re-label the elements of $\Y$ so that $i_1=i_n=1$. By \ref{HMMcon1} there must exist $\epsilon>0$ such that $\mu(A_j)>0$ for each state $j\in \Y$, where
\begin{align*}
A_j \DEF \left\{x\in\mathcal{X} \: | \: f_j(x)p_{ \Cdot j}(1-\epsilon)> \max_{i\in \Y,~i\ne
j}f_i(x)p_{\Cdot i}\right\}.
\end{align*}
Denote $A=A_{i_2} \times A_{i_3} \times \cdots \times A_{i_n}$. Next we show that for every $x_{1:n} \in \X \times A$
\begin{align}
&p_{11}(x_{1:n}) \geq p_{i1}(x_{1:n}),\quad \forall i \in \mathcal{Y}, \label{A1HMM1}\\
&p_{11}(x_{1:n}) > p_{1i}(x_{1:n}) ,\quad \forall i \in \mathcal{Y}\setminus \{1\},\label{A1HMM2}\\
& p_{11}(x_{1:n})(1-\epsilon) > p_{ij}(x_{1:n}), \quad \forall i,j \in \mathcal{Y} \setminus \{1\}. \label{A1HMM3}
\end{align}
Indeed, by construction of $i_{1:n}$ and $A$, for any path $y_{1:n}$ for which $y_{2:n} \neq i_{2:n}$ and for any $x_{2:n} \in A$ we have
\begin{align*}
(1-\epsilon)\prod_{k=2}^n p_{i_{k-1}i_k}f_{i_k}(x_k)>\prod_{k=2}^n p_{y_{k-1}y_k}f_{y_k}(x_k).
\end{align*}
On the other hand, for \textit{any} path $y_{1:n}$ and any $x_{2:n} \in  A$
\begin{align*}
\prod_{k=2}^n p_{i_{k-1}i_k}f_{i_k}(x_k) \geq \prod_{k=2}^n p_{y_{k-1}y_k}f_{y_k}(x_k).
\end{align*}
Thus the inequalities (\ref{A1HMM1}), (\ref{A1HMM2}) and (\ref{A1HMM3}) must hold.

Note now that there must exist $0 < \delta_0 \leq \Delta_0< \infty$ such that, defining $G_i^0=\{x \in \X \:| \: \delta_0 \leq f_i(x); \: f_j(x) \leq \Delta_0, j \in \Y\}$, we have  $\mu(G_i^0)>0$ for every state $i \in \Y$. Furthermore, denoting $G=\left(\cap_{i \in C}G_i^0\right) \setminus \left(\cup _{i\notin C}G_i \right)$, by cluster assumption we may with no loss of generality assume that $\delta_0$ is so small and $\Delta_0$ is so large that $\mu(G)>0$. Fix $j' \in C$. By irreducibility assumption there exists path $u_{1:K}$, $K \geq 2$, such that $u_1=j'$, $u_K=1$ and $p_{u_{k-1}u_k}>0$ for all $k=2, \ldots K$. Similarly, there exists path $v_{1:L}$, $L \geq 3$, such that $v_1=1$, $v_L=j'$ and $p_{v_{k-1}v_k}>0$ for all $k=2, \ldots L$. Denote $H_1=G^0_{u_2} \times \cdots \times G^0_{u_K}$, $H_2=G^0_{v_2} \times \cdots \times G^0_{v_{L}}$ and $p^*=\min\{p_{ij} \:|\: p_{ij}>0, \: i,j \in \Y \}$. With no loss of generality we may assume that $\delta_0<1$ and $\Delta_0>1$. Denote $M=K \vee L+R+1$, where $\vee$ denotes maximum, and set $\delta=(p^*\delta_0)^{M}$, $\Delta=\Delta_0^{M}$ and $N \geq 2$ so big that \textbf{A3} holds. Take
\begin{align*}
\X^*=\X \times G^{R+1} \times H_1 \times A^{2N} \times H_2 \times G^{R},
\end{align*}
$n_1=R+1+K$, $n_k=n_{k-1}+n-1$ for $k=2,\ldots,2N+1$, and $n_{2N+2}=n_{2N+1}+L-1+R$. By (\ref{A1HMM1}), (\ref{A1HMM2}) and (\ref{A1HMM3}) \textbf{A1'} holds.

Next, we will prove \textbf{A2}.
First note that by definition of sets $G_i^0$,
\begin{align*}
&p_{ij}(x) \leq \Delta, \quad \forall i,j \in \mathcal{Y}, \quad \forall  x \in \mathcal{X}^*_{(1,n_1)} \cup \mathcal{X}^*_{(n_{2N+1},n_{2N+2})}.
\end{align*}
Next, denote $\Y_C=\{i \in \Y \:| \: p_{ij}>0, \: j \in C\}$. Note that by definition of set $G$, $\mathcal{Y}^+(x)_{(1)} \subset\Y_C$ for all $x \in  \mathcal{X}^*_{(1,n_1)}$. By the primitiveness of $\mathbb{P}_C$, we have for all $x_{1:n_1} \in  \mathcal{X}^*_{(1,n_1)}$ and any $i \in \Y_C$
\begin{align*}
\quad p_{i1}(x_{1:n_1})  \geq \max_{y_{1:n_1} \colon y_1=i, \: y_{2:R+1} \in C^{R}, \: y_{R+2:n_1}=u_{1:K}}p(x_{1:n_1},y_{1:n_1}) \geq \delta>0.
\end{align*}
Also note that by definition of sets $G$, $\mathcal{Y}^+(x)_{(2)} \subset C$ for all $x \in  \mathcal{X}^*_{(n_{2N+1},n_{2N+2})}$. Denoting $n=n_{2N+1}$ and $n'=n_{2N+2}$, we have by the primitiveness of $\mathbb{P}_C$ for all $x_{n:n'} \in  \mathcal{X}^*_{(n,n')}$ and any $j \in C$
\begin{align*}
p_{1j}(x_{n:n'}) \geq \max_{y_{n:n'} \colon y_{n:n+L-1}=v_{1:L} , \: y_{n+L:n'} \in C^R, \:y_{n'}=j } p(x_{n:n'},y_{n:n'})\geq \delta>0.
\end{align*}
The arguments above show that \textbf{A2} must hold and that
\begin{align} \label{Yplusx}
\Y^+(x)_{(1)}=\Y_C, \quad \forall x \in
\X^*_{(1,n_1)}.
\end{align}
From \eqref{Yplusx}, Lemma \ref{lemio} and Lemma \ref{HMMHarris} it follows that $P(X \in \X^* \mbox{ i.o.})=1$.
\end{proof}

\subsection{Discrete $\X$}\label{disc}
Consider the case where $\X$ is discrete (finite or
countable) and $Z$ is an irreducible and recurrent Markov chain with
(discrete) state space ${\cal Z}'\subset {\cal X}\times {\cal
Y}$. Here the state-space refers to the set of possible values of
$Z$. Note that ${\cal Z}'$ can be a proper subset ${\cal X}\times
{\cal Y}$. Also note: since the transition kernel $q(z|z')$ is defined on ${\cal Z}'$,
the definition of $\Y^+(x_{1:q})$ immediately implies that
$(i,x_1)\in {\cal Z}'$ for every $i\in \Y^+(x_{1:q})_{(1)}$.
The following simple result can
be derived from Theorem \ref{thLSC}.
\begin{cor} \label{discCor} Let $\X$ be discrete and let $Z$ be an
irreducible and recurrent Markov chain with the state-space ${\cal
Z}'\subset {\cal X}\times {\cal Y}$. Then the following
conditions ensure that there exists a barrier set $\X^*$ consisting
of strong 1-barriers of fixed order and satisfying
 $P(X \in \X^* \mbox{ i.o.})=1$.
\begin{enumerate}[label=(\roman*)] \item \label{discClust} There exists $q \geq 2$ and a sequence $x_{1:q} \in \X^q$ such that
$\Y^+(x_{1:q})_{(1)}$ is non-empty and $(i,j) \in \Y^+(x_{1:q})$ for
every $i \in \Y^+(x_{1:q})_{(1)}$ and $j \in \Y^+(x_{1:q})_{(2)}$.
\item \label{discCyc} There exists $n \geq 2$ and $x^*_{1:n} \in \X^n$ such that $(x^*_1,1)\in {\cal Z}'$
and
\begin{enumerate}
\item[1. ] it holds
\begin{align*}
x_1^*=x_n^* \quad \mbox{and} \quad  p_{11}(x^*_{1:n}) > p_{ij}(x^*_{1:n}), \quad \forall i,j \in \mathcal{Y} \setminus \{1\};
\end{align*}
\item[2. ] it holds
\begin{align} \label{discIneq1}
&p_{11}(x^*_{1:n}) > p_{i1}(x^*_{1:n}),\quad \forall i \in \mathcal{Y},\\
&p_{11}(x^*_{1:n}) > p_{1i}(x^*_{1:n}) ,\quad \forall i \in
\mathcal{Y},\label{discIneq2}
\end{align}
where either inequalities \eqref{discIneq1} or inequalities \eqref{discIneq2} could be non-strict.
\end{enumerate}
\end{enumerate}
\end{cor}
\begin{proof}
The proof is straightforward application of Theorem \ref{thLSC}. To formally apply Theorem \ref{thLSC}, $Z$ should be viewed as a Markov chain on product space $\Z=\Z \times \Y$. In that perspective $Z$ may no longer be irreducible. However with no loss of generality we may assume that $Z$ is $\psi$-irreducible and Harris recurrent, where support of $\psi$ is $\Z'$. Indeed, this can be achieved by fixing $(x',i') \in \Z'$ and taking $q(x',i'|z)=1$ for all $z \in \Z \setminus \Z'$. Then all elements of $\Z'$ are reachable. Also, assuming with no loss of generality that $x' \neq x_2$, we have that $\Y^+(x_{1:q})$ is the same regardless if it is defined on the product space $\Z$ or subspace $\Z'$. Next, simply take $E=\{x_{1:q}\}$ so that \ref{Bcluster} holds with
$E_{(1)}=\{x_1\}$ and $i_E$ being any element of
$\Y^+(x_{1:q})_{(1)}$. To see that \ref{Bcenter} holds, denote
$x^*=x^*_{1:n-1}$ and note that for arbitrary $N \geq 2$ the strong
center part of the barrier set can be taken to be
\begin{align*}
\X^*_{(n_1:n_{2N+1})}=\{(\underbrace{x^*, x^*, \ldots, x^*}_{\text{$2N$ blocks of $x^*$}}, x^*_n)\}.
\end{align*}
In the discrete case measure $\mu$ is counting measure
on $2^\X$, which is strictly positive, and the functions $(x',x)
\mapsto q(x,j|x',i)$ are always continuous and bounded. Therefore
Theorem \ref{thLSC} applies.
\end{proof}

\noindent
{\bf Remarks about the condition (i).}
\begin{enumerate}
  \item If $Z$ is stationary MC, then the set $\Y^+(x_{1:q})_{(1)}$ consists
of states $i$ satisfying the following property: there exists
$y_{1:q}\in {\cal Y}^q$ such that $y_1=i$ and
$p(x_{1:q},y_{1:q})>0$. Similarly $\Y^+(x_{1:q})_{(2)}$ consists
of states $j$ satisfying the following property: there exists
$y_{1:q}\in {\cal Y}^q$ such that $y_q=j$ and
$p(x_{1:q},y_{1:q})>0$. However, given $i\in
\Y^+(x_{1:q})_{(1)}$ and $j\in \Y^+(x_{1:q})_{(2)}$, there need
not necessary be any path $y_{1:q}$ beginning with $i$ (i.e.
$y_1=i$) and ending with $j$ (i.e. $y_q=j$) such that
$p(x_{1:q},y_{1:q})>0$. The condition (i) ensures that  for
every pair  $i\in \Y^+(x_{1:q})_{(1)}$ and $j\in
\Y^+(x_{1:q})_{(2)}$ such a path exists and then $(i,j)\in
\Y^+(x_{1:q}).$ Interestingly, in ergodic theory, this property
is the same as the {\it subpositivity} of the word $x_{1:q}$ for
{\it factor map} $\pi: {\cal Z}\to {\cal X}, \pi(x,y)=x$, see
(\cite{yoo}, {Def 3.1}). Thus (i) ensures that a.e. realization
of $X$ process has infinitely many subpositive words.
  \item Let us now
argue that for stationary $Z$, the subpositivity  is also very
close to be a necessary property of a barrier. Indeed, if
$x_{k:l}$ ($1<k<l<n$) is a barrier containing a strong 1-node,
then for any Viterbi path $v(x_{1:n})$, $(v_{k},v_{l})\in
\Y^+(x_{k:l})$. Suppose now there exists another words of
observations $x'_{1:k-1}$ and $x'_{l+1:n}$ such that the
corresponding Viterbi path $v'=v(x'_{1:k-1},x_{k:l},x'_{l+1:n})$
satisfies: $v'_k\ne v_k$ and $v'_l\ne v_l$. Then also
$(v'_{k},v'_{l})\in \Y^+(x_{k:l})$. Take now $v'_{k}\in
\Y^+(x_{k:l})_{(1)}$ and $v_l\in \Y^+(x_{k:l})_{(2)}$ and ask:
does $(v'_{k},v_l)\in \Y^+(x_{k:l})$? Since  $x_{k:l}$ is  a
barrier containing a strong 1-node, then by piecewise
construction there exists a Viterbi path
$w=v(x'_{1:k-1},x_{k:n})$ such that $w_k=v'_k$ and $w_l=v_l$ and
so $(v'_{k},v_l)\in \Y^+(x_{k:l})$. We have seen that if  $i\in
\Y^+(x_{k:l})_{(1)}$ is  such that for some $x'_{1:k-1}$,
$v_k(x'_{1:k-1},x_{k:n})=i$ and if $j\in \Y^+(x_{k:l})_{(2)}$ is
such that for some $x'_{l+1:n}$, $v_l(x_{1:l},x'_{l+1:n})=j$,
then $(i,j)\in \Y^+(x_{k:l})$. Therefore, if every $i\in
\Y^+(x_{k:l})_{(1)}$ and every $j\in \Y^+(x_{k:l})_{(2)}$
satisfies above-stated property of being included into a Viterbi
path (and often this is the case), then (i) and also {\bf B1} is
a  necessary property of a barrier.
\end{enumerate}

\begin{example}
Let $\mathcal{X}=\mathcal{Y}=\{1,2\}$, and assume that $X$ and $Y$ are Markov chains both having the transition matrix $\begin{pmatrix}
  p & 1-p \\
      q & 1-q \\
\end{pmatrix}$,
where $p, q \in (0,1)$. Then, as is shown in \cite{MomentBound}, the transition matrix of $Z$ has the form
\begingroup
\renewcommand*{\arraystretch}{1.3}
\renewcommand{\kbldelim}{(}
\renewcommand{\kbrdelim}{)}
\begin{equation*} 
\mathbb{Q} \DEF
\kbordermatrix{  & (1,1) & (1,2) & (2,1) & (2,2)\\
     (1,1) & p\lambda_1 & p(1-\lambda_1) & p(1-\lambda_1) & 1+p\lambda_1-2p \\
     (1,2) & p\lambda_2 & p(1-\lambda_2) & q-p\lambda_2 & 1+p\lambda_2-q-p \\
     (2,1) & q\mu_1 & q(1-\mu_1) & p-q\mu_1 & 1+q\mu_1-p-q \\
     (2,2) & q\mu_2 & q(1-\mu_2) & q(1-\mu_2) & 1+q\mu_2-2q \\
    },
\end{equation*}
\endgroup
where
\begin{align} \label{LamCond}
&\lambda_1 \in \left[{2p-1\over p}\vee 0,1 \right],\quad
\lambda_2 \in \left[{q+p-1\over p}\vee 0,{q \over p}\wedge 1 \right], \\
&\mu_1\in \left[{p+q-1\over q}\vee 0,{p\over q}\wedge 1 \right],
\quad \mu_2 \in \left[{2q-1\over q}\vee 0,1 \right]. \label{MuCond}
\end{align} Here $\vee$ and $\wedge$ denote the maximum and minimum,
respectively. When $Z$ is stationary, then $X$ and $Y$ are
independent if and only if
$$\lambda_1=\mu_1=p,\quad
\lambda_2=\mu_2=q.$$ 

Assume now that $\lambda_i$ and $\mu_i$ are not allowed to have the
extreme values of the constraints (\ref{LamCond}) and
(\ref{MuCond}). Then the elements of $\mathbb{Q}$ are positive,
which implies that $\Z'={\cal Z}$  and $Z$
is irreducible and recurrent. Also  \ref{discClust} of Corollary
\ref{discCor} trivially holds for any $x_{1:q}\in \{1,2\}^q$, where
$q\geq 2$. Thus the existence of infinitely many strong nodes for
almost every realization of $X$ is guaranteed if  there exists
$x^*_{1:n} \in \{1,2\}^n$, $n \geq 2$, such that
\begin{align} \label{B3cond}
x^*_1=x^*_n \quad \mbox{and} \quad p_{11}(x^*_{1:n})> p_{ij}(x^*_{1:n}), \quad \forall (i,j) \in \Y^2 \setminus \{(1,1)\}.
\end{align}
Taking $x^*_{1:n}=(1,1)$, we have that \eqref{B3cond} holds whenever
\begin{align*}
&p\lambda_1 > \max\{p(1-\lambda_1), p \lambda_2, p(1-\lambda_2)\}\\
&\Leftrightarrow  \quad \lambda_1 > \max\{1-\lambda_1,  \lambda_2, 1-\lambda_2\}\\
&\Leftrightarrow  \quad \lambda_1 >  \lambda_2 \vee (1-\lambda_2) .
\end{align*}
Taking $x^*_{1:n}=(2,2)$, we have that \eqref{B3cond} holds when
\begin{align*}
p-q\mu_1 > \max\{1+q\mu_1-p-q,q(1-\mu_2),1+q\mu_2-2q\}.
\end{align*}
Switch now the labels of $\Y$. Taking $x^*_{1:n}=(1,1)$, we obtain that \eqref{B3cond} holds when
\begin{align*}
&p(1-\lambda_2) > \max\{ p\lambda_1, p(1-\lambda_1), p\lambda_2)\}\\
&\Leftrightarrow  \quad 1-\lambda_2 > \max\{ \lambda_1, 1-\lambda_1, \lambda_2\}\\
&\Leftrightarrow  \quad -\lambda_2 > \max\{ \lambda_1-1, -\lambda_1, \lambda_2-1\}\\
&\Leftrightarrow  \quad \lambda_2 < (1-\lambda_1) \wedge \lambda_1.
\end{align*}
Taking $x^*_{1:n}=(2,2)$, we have that \eqref{B3cond} holds when
\begin{align*}
1+qp_2-2q> \max \{p-q \mu_1, 1+q\mu_1-p-q, q(1-\mu_2)\}.
\end{align*}
Further conditions can be found with $n=3,4, \ldots$
\end{example}

\subsection{Linear Markov switching model}\label{sec:LMSW}
Let $\X= \mathbb{R}^d$ for some $d \geq 1$ and for each state $i \in \Y$ let $\{\xi_k(i)\}_{k \geq 2}$ be an i.i.d. sequence of random variables on $\X$ with $\xi_2(i)$ having density $h_i$ with respect to Lebesgue measure on $\mathbb{R}^d$. We consider the ``linear Markov switching model'', where $X$ is defined recursively by
\begin{align}\label{LMSM}
X_{k}=F(Y_k)X_{k-1}+ \xi_k(Y_k), \quad k \geq 2.
\end{align}
Here $F(i)$ are some $d \times d$ matrices, $Y=\{Y_k\}_{k\geq 1}$ is a Markov chain with transition matrix $(p_{ij})$, $X_1$ is some random variable on $\X$, and random variables $\{\xi_k(i)\}_{k \geq 2, \: i \in \Y}$ are assumed to be independent and independent of $X_1$ and $Y$. Recall that for Markov switching model, the transition density expresses as $q(x,j|x',i)=p_{ij}f_j(x|x')$. For the current model measure $\mu$ is Lebesgue measure on $\mathbb{R}^d$ and $f_j(x|x')=h_j(x-F(j)x')$. When $F(i)$ are zero-matrices, then the linear Markov switching model simply becomes HMM with $h_i$ being the emission densities. When $d=1$, we obtain the ``switching linear autoregression of order 1''. The switching linear autoregressions are popular in econometric modelling, see e.g. \cite{HMMbook} and the references therein.

We will now apply Theorem \ref{thLSC} to the linear Markov switching model. The requirement (of Theorem \ref{thLSC}) that $\mu$ must be strictly positive is trivially fulfilled in the case where $\mu$ is Lebesgue measure. The requirement that functions $(x',x) \mapsto q(x,i|x',j)$ must be lower semi-continuous and bounded is fulfilled when $h_j$ are lower semi-continuous and bounded (composition of lower semi-continuous function with continuous function is lower semi-continuous). Deriving simple conditions which ensure \ref{Bcluster} is also quite easy. In what follows, let $\| \cdot \|$ denote the 2-norm on $\X=\mathbb{R}^d$, and for any $x \in \X$ and $r>0$ let $B(x,r)$ denote an open ball in $\X$ with respect to 2-norm with center point $x$ and radius $r>0$.

\begin{lemma}\label{LMB1B2Lemma} Let $Z$ be the linear Markov switching model. If the following conditions are fulfilled, then $Z$ satisfies \ref{Bcluster}.
\begin{enumerate}[label=(\roman*)]
\item  \label{LMprimit}There exists set $C\subset \Y$ and $r>0$ such that the following two conditions are satisfied:
\begin{enumerate}
\item[1.] for $x \in B(0,r)$, $h_i(x)>0$ if and only if $i \in C$;
  \item[2.] the sub-stochastic matrix $\mathbb{P}_C= (p_{ij})_{i,j\in C}$ is primitive, i.e. there exists $R \geq 1$ such that matrix $\mathbb{P}_C^R$ has only positive elements.
\end{enumerate}
\item \label{LMreach} Denote $\Y_C=\{i \in \Y \:| \: p_{ij}>0, \: j\in C\}$. There exists $i_E \in \Y_C$ such that $(0,i_E)$ is reachable.
\end{enumerate}
\end{lemma}

Conditions \ref{LMprimit} and \ref{LMreach} are not very restrictive. For example, when all the elements of $\mathbb{P}$ are positive, then \ref{LMprimit} is fulfilled if densities $h_i$ are either positive around 0 or zero around 0 and there exists at least one $j \in \Y$ such that $h_j$ is positive around 0. If densities $h_i$ are all positive around 0, then \ref{LMprimit} is fulfilled when $\mathbb{P}$ is primitive with $C=\Y$. If $h_i$ are positive everywhere and $Y$ is irreducible, then all points in $\Z$ are reachable and so \ref{LMreach} trivially holds. 
\begin{proof}[Proof of Lemma \ref{LMB1B2Lemma}.] There must exist $r_0 >0$ such that
\begin{align}\label{lessr}
\|x-F(j)x'\|<r, \quad \forall j \in \Y, \quad \forall x,x' \in B(0,r_0).
\end{align}
By \ref{LMprimit} there exists $R \geq 1$ such that $\mathbb{P}_C^R$ contains only positive elements. We take $E=B(0,r_0)^{R+2}$. Fixing $x_{1:R+2} \in E$, we have for any $i,j \in \Y$
\begin{align*}
p_{ij}(x_{1:R+2})&=\max_{y_{1:R+2} \colon (y_1,y_{R+2})=(i,j)}\prod_{k=2}^{R+2}p_{y_{k-1}y_k}h_{y_k}(x_k-F(y_k)x_{k-1}).
\end{align*}
Together with \eqref{lessr} and \ref{LMprimit} this implies that $p_{ij}(x_{1:R+2})>0$ if and only if $i \in \Y_C$ and $j \in C$. Hence $\Y^+(x)=\Y_C \times C$ for every $x \in E$. Together with \ref{LMreach} this implies that \ref{Bcluster} holds with $x_E=0$.
\end{proof}

As for the condition \ref{Bcenter}, the following lemma provides one possible way to construct the center part of the barrier set.
\begin{lemma} \label{LMB3}Let $Z$ be the linear Markov switching model. If the following condition is fulfilled, then $Z$ satisfies \ref{Bcenter}: there exists $x^* \in \X$ such that
\begin{enumerate}[label=(\roman*)] \item \label{LMB3p11}
$p_{11}=\max_{i \in \Y}p_{i1}$;
\item \label{LMB3reach}$(x^*,1)$ is reachable;
\item \label{LMB3cycle}$h_i$ is continuous at $x^*-F(i)x^*$ for all $i \in \Y$, and
\begin{align*}
p_{11}h_1(x^*-F(1)x^*)>p_{ij}h_j(x^*-F(j)x^*), \quad \forall i \in \Y, \quad \forall j \in \Y \setminus \{1\}.
\end{align*}
\end{enumerate}
\end{lemma}
\begin{proof} By \ref{LMB3cycle} there must exist $\epsilon>0$ and $r>0$ such that
\begin{multline} \label{cycle2}
p_{11}h_1(x)(1-\epsilon)>p_{ij}h_j(x'), \quad \forall i \in \Y, \quad \forall j \in \Y \setminus \{1\}, \\
\forall x \in B(x^*-F(1)x^*,r), \quad \forall x' \in B(x^*-F(j)x^*,r).
\end{multline}
Also there must exist $r'>0$ such that
\begin{align} \label{cycle3}
\|x -F(j)x' - (x^* -F(j)x^*) \|<r, \quad \forall j \in \Y, \quad \forall x,x' \in B(x^*,r').
\end{align}
For some $N \geq 2$ we define the center part of the barrier set by
\begin{align*}
\X^*_{(n_1,n_{2N+1})}=B(x^*,r')^{2N+1}
\end{align*}
We confirm that $\X^*_{(n_1,n_{2N+1})}$ is indeed a strong center part of a barrier set, i.e. that it satisfies \textbf{A1'}. We take $n_{k+1}-n_k=1$ for all $k=1, \ldots,2N$. Let $x',x \in B(x^*,r')$. We have for all $i \in \Y$ and $j \in \Y \setminus \{1\}$
\begin{align*}
p_{11}(x',x)(1-\epsilon)=p_{11}h_1(x-F(1)x')(1-\epsilon)> p_{ij}h_{j}(x-F(j)x')=p_{ij}(x',x).
\end{align*}
Here the inequality follows from \eqref{cycle3} and \eqref{cycle2}. On the other hand we have by \ref{LMB3p11} for all $i \in \Y$
\begin{align*}
p_{11}(x',x)=p_{11}h_1(x-F(1)x')\geq p_{i1}h_1(x-F(1)x')=p_{i1}(x',x).
\end{align*}
The arguments above show that \textbf{A1'} does indeed hold. Hence by \ref{LMB3reach} \ref{Bcenter} holds.
\end{proof}
For the sake of simplicity Lemma \ref{LMB3} uses only cycles of length 2 in the construction of barrier set, but this could easily be generalized to include cycles of arbitrary length.

\begin{remark} In the proofs of Lemmas \ref{LMB1B2Lemma} and \ref{LMB3} the specific structure of the linear Markov switching model has not played a very big role, so a natural question is, if analogous results could be proven for more general models. More specifically, we can consider a Markov switching model, where instead of recursion \eqref{LMSM} $X$ is more generally defined by
\begin{align} \label{LMSGen}
X_{k}=G(Y_k, X_{k-1})+ \xi_k(Y_k), \quad k \geq 2,
\end{align}
where $G(i,\cdot) \colon \mathbb{R}^d \rightarrow \mathbb{R}^d$ are some continuous functions. For this model, the transition kernel density expresses as $q(x,j|x',i)=p_{ij}h_j(x-G(j,x'))$. The statement of Lemma \ref{LMB3} indeed holds for this model, if we replace the condition \ref{LMB3cycle} with the following generalized version: $h_i$ is continuous at $x^*-G(i, x^*)$ for all $i \in \Y$, and
\begin{align*}
p_{11}h_1(x^*-G(1, x^*))>p_{ij}h_j(x^*-G(j,x^*)), \quad \forall i \in \Y, \quad \forall j \in \Y \setminus \{1\}.
\end{align*}
The statement of Lemma \ref{LMB1B2Lemma} also holds for model \eqref{LMSGen}, if we demand that the $G(i,\cdot)$ satisfy the following additional condition:
\begin{align} \label{GCond}
G(i,0)=0, \quad \forall i \in \Y.
\end{align}
If \eqref{GCond} is too restrictive, a different approach is needed to prove \ref{Bcluster}. In any case, if $h_i$ are everywhere positive and $\mathbb{P}$ is primitive, then, as it is easy to verify, \ref{Bcluster} holds regardless of whether \eqref{GCond} holds or not.
\end{remark}

It remains to address the issue of Harris recurrence of the linear Markov switching model. In what follows, for $x \in \X$ we denote with $\| x \|_1$ the 1-norm of $x$, and for a $d \times d$ matrix $A$ we denote with $\| A \|_1$ the 1-norm of matrix $A$, that is $\|A\|_1$ is the maximum absolute column sum of $A$.
\begin{lemma} \label{LMHarris} Let $Z$ be the linear Markov switching model. If the following conditions are fulfilled, then $Z$ is Harris recurrent:
\begin{enumerate}[label=(\roman*)] \item $Z$ is $\psi$-irreducible and support of $\psi$ has non-empty interior;
\item $\mathbb{E}\|\xi_2(i) \|_1< \infty$ for all $i \in \Y$;
\item $\max_{i \in \Y}\sum_{j \in \Y}p_{ij}\|F(j)\|_1<1$.
\end{enumerate}
\end{lemma}

Proof of this statement is given in Appendix \ref{LMHarrisProof}.

Applying the results above to the case where $h_i$ are Gaussian yields
\begin{cor} \label{GLMcor} Let $Z$ be the linear Markov switching model, with densities $h_i$ being Gaussian with respective mean vectors $\mu_i$ and positive definite covariance matrices $\Sigma_i$. If the following conditions are fulfilled, then there exist a barrier set $\X^*$ consisting of strong $1$-barriers of fixed order and satisfying $P(X \in \X^* \mbox{ i.o.})=1$.
\begin{enumerate}[label=(\roman*)]\item \label{GLMprim} Matrix $\mathbb{P}=(p_{ij})$ is primitive, i.e. there exists $R$ such that $\mathbb{P}^R$ consists of only positive elements.
\item \label{p11Max} It holds $p_{11}=\max_{i \in \Y}p_{i1}$.
\item \label{GLMcycle}Matrix $\mathbb{I}_d - F(1)$, where $\mathbb{I}_d$ denotes the identity matrix of dimension $d$, is non-singular, and for all $i \in \Y$ and $j \in \Y \setminus \{1\}$
\begin{align*}
(\mathbb{I}_d-F(j))(\mathbb{I}_d-F(1))^{-1}\mu_1 \in \mathbb{R}^d \setminus H_{ij},
\end{align*}
where
\begin{align*}
H_{ij} \DEF \begin{cases} \emptyset, &\mbox{if $p_{ij}=0$ or $\dfrac{p_{11}\sqrt{|\Sigma_j |}}{p_{ij} \sqrt{|\Sigma_1|}}>1$}, \\
\left\{x \in \mathbb{R}^d \: | \: (x- \mu_j)^\top \Sigma_j^{-1}(x- \mu_j) \leq  -2 \ln \left(\dfrac{p_{11}\sqrt{|\Sigma_j |}}{p_{ij} \sqrt{|\Sigma_1|}} \right)  \right\}, &\mbox{else}
\end{cases}.
\end{align*}
\item \label{GLMHarris}It holds $\max_{i \in \Y}\sum_{j \in \Y}p_{ij}\|F(j)\|_1<1$.
\end{enumerate}
\end{cor}
\begin{proof} By \ref{GLMprim} $Y$ is irreducible. This together with the fact that densities $h_i$ are positive on the whole space $\mathbb{R}^d$ implies that all elements in $\Z=\mathbb{R}^d \times \Y$ are reachable and that $Z$ is $\mu \times c$-irreducible, where $\mu$ denotes the Lebesgue measure on $\mathbb{R}^d$ and $c$ denotes the counting measure on $\Y$. It follows from \ref{GLMprim} and Lemma \ref{LMB1B2Lemma} that \ref{Bcluster} holds. Take now $x^*=(\mathbb{I}_d-F(1))^{-1}\mu_1$ (then $x^*-F(1)x^*=\mu_1$ and so $x^*-F(1)x^*$ maximises $h_1$). Condition \ref{GLMcycle} implies
\begin{align} \label{notinHij}
x^*-F(j)x^* \notin H_{ij}, \quad \forall i \in \Y, \quad \forall j \in \Y \setminus \{1\}.
\end{align}
Some calculation reveals that $\{x \in \mathbb{R}^d \: | \: p_{ij}h_j(x) \geq p_{11}h_1(x^*-F(1)x^*)\}=H_{ij}$ and so \eqref{notinHij} implies
\begin{align*}
p_{ij}h_j(x^*-F(j)x^*)<p_{11}h_1(x^*-F(1)x^*), \quad \forall i \in \Y, \quad \forall j \in \Y \setminus \{1\}.
\end{align*}
This together with assumption \ref{p11Max} and Lemma \ref{LMB3} implies that \ref{Bcenter} holds. By \ref{GLMHarris} and Lemma \ref{LMHarris} $Z$ is Harris recurrent, so the statement follows from Theorems \ref{th1} and \ref{thLSC}.
\end{proof}
In some cases the condition \ref{p11Max} of Corollary \ref{GLMcor} can be rather restrictive, particularly when the diagonal entries of $\mathbb{P}=(p_{ij})$ are small and so there are not many (or none at all) diagonal entries of $\mathbb{P}$ which dominate their column (i.e. are larger than or equal to other column entries). In that case one possible solution is to group the elements of $Z$ to pairs, that is consider the model $Z'=\{((X_{2k-1},X_{2k}),(Y_{2k-1},Y_{2k}))\}_{k\geq 1}$ instead of $Z=\{(X_k,Y_k)\}_{k\geq 1}$. The transition kernel density of chain $Z'$ is simply
\begin{align*}
q'(z_3,z_4|z_1,z_2) \colon ((z_1,z_2),(z_3,z_4)) \mapsto p(z_3,z_4|z_2)=q(z_3|z_2)q(z_4|z_3).
\end{align*}
Let $X'$ and $Y'$ denote the marginals of $Z'$:
$X'=\{(X_{2k-1},X_{2k})\}_{k \geq 1}$ and
$Y'=\{(Y_{2k-1},Y_{2k})\}_{k \geq 1}$. The existence of Viterbi
process  for $X'$ implies the existence of Viterbi process for $X$ under appropriate tie-breaking rules. Indeed, consider the case where the tie-breaking scheme corresponding to $X'$ is lexicographic, induced by the following ordering on $\Y^2$:
\begin{align*}
(1,1) \succ (1,2) \succ (1,3)  \succ \cdots \succ (2,1) \succ (2,2) \succ (2,3) \succ \cdots \succ (|\Y| , |\Y|).
\end{align*}
We also assume that the ordering on $\Y$ is
\begin{align*}
1 \succ 2 \succ \cdots \succ |\Y|.
\end{align*}
Thus $\Y^2$ is equipped with lexicographic ordering induced by the ordering on $\Y$. We assume that the tie-breaking scheme corresponding to $X$ is lexicographic as well. Then, if $\{(V_{2k-1},V_{2k})\}_{k\geq 1}$ is the Viterbi process of $X'$, $\{V_k\}_{k \geq 1}$ is the Viterbi process of $X$.

Chain $Z'$ is a linear Markov switching model on space $\mathbb{R}^{2d} \times \Y'$, where $\Y' \DEF \{(i,j) \in \Y^2 \:| \: p_{ij}>0\}$. To see this, note that
\begin{align*}
\begin{pmatrix} X_{2k-1}\\X_{2k}
  \end{pmatrix}
 = \begin{pmatrix} \mathbf{0} &F(Y_{2k-1})\\
 \mathbf{0} &F(Y_{2k})F(Y_{2k-1})
  \end{pmatrix}
  \begin{pmatrix} X_{2k-3}\\X_{2k-2}
  \end{pmatrix}
  + \xi'_{k}(Y_{2k-1},Y_{2k}), \quad k\geq 2,
\end{align*}
where
\begin{align*}
\xi'_k(i,j) \DEF \begin{pmatrix} \xi_{2k-1}(i)\\
  F(j)\xi_{2k-1}(i)+\xi_{2k}(j)
  \end{pmatrix} =
  \begin{pmatrix}
 \mathbb{I}_d & \mathbf{0}\\
 F(j) &\mathbb{I}_d
  \end{pmatrix}
\begin{pmatrix} \xi_{2k-1}(i)\\
  \xi_{2k}(j)
  \end{pmatrix}.
\end{align*}
The matrix 
$B \DEF \begin{pmatrix}
 \mathbb{I}_d & \mathbf{0}\\
 F(j) &\mathbb{I}_d
  \end{pmatrix}$
has full rank, so assuming that $\xi_k(i)$ are non-degenerate Gaussian with respective mean vectors $\mu_i$ and covariance matrices $\Sigma_i$, then random vectors $\xi_k'(i,j)$ are non-degenerate Gaussian with respective mean values $\begin{pmatrix}
\mu_i\\
F(j)\mu_i+\mu_j
\end{pmatrix}$
and covariance matrices $B\begin{pmatrix}
\Sigma_i &\mathbf{0} \\
\mathbf{0} &\Sigma_j
\end{pmatrix}B^{\top}$. Therefore Corollary \ref{GLMcor} applies to both $Z$ and $Z'$.

Markov chain $Y'$, the hidden process of model $Z'$, has transition matrix $\mathbb{P}' \DEF (p_{jk}p_{kl})_{(i,j),(k,l) \in \Y'}$. Matrix $\mathbb{P}'$ might have diagonal entries which dominate their column, even if $\mathbb{P}$ does not have such entries. As a simple example consider the case where $\mathbb{P}=\begin{pmatrix}
\epsilon & 1-\epsilon \\
1-\epsilon' & \epsilon'
\end{pmatrix}$, where $\epsilon, \epsilon' \in (0, \frac{1}{2})$. Then
\begingroup
\renewcommand*{\arraystretch}{1.3}
\renewcommand{\kbldelim}{(}
\renewcommand{\kbrdelim}{)}
\begin{equation*} \mathbb{P'}=
\kbordermatrix{  & (1,1) & (1,2) & (2,1) & (2,2)\\
     (1,1) & \epsilon^2            & \epsilon(1-\epsilon)      &  (1- \epsilon)(1-\epsilon')&(1-\epsilon)(\epsilon')^2 \\
     (1,2) & (1-\epsilon')\epsilon & (1-\epsilon')(1-\epsilon) & \epsilon'(1-\epsilon')    &(\epsilon')^2 \\
     (2,1) & \epsilon^2            & \epsilon(1-\epsilon)      & (1- \epsilon)(1-\epsilon') &(1-\epsilon)(\epsilon')^2 \\
     (2,2) & (1-\epsilon')\epsilon & (1-\epsilon')(1-\epsilon) & \epsilon'(1-\epsilon')       & (\epsilon')^2 \\
    },
\end{equation*}
\endgroup
and so the second and third diagonal entry of $\mathbb{P}'$ dominates its column.

In general primitiveness of $\mathbb{P}$ does not imply primitiveness of $\mathbb{P}'$, but the reader can easily verify that if there exists an \textit{odd} positive integer $R$ such that $\mathbb{P}^R$ contains only positive entries, then $\mathbb{P}'^{(R+1)/2}$ contains only positive entries and so $\mathbb{P}'$ is primitive. We also note that the approach described above can easily be generalized to include groupings of triplets, quadruplets, etc.

\begin{appendices}
\section{Supporting results}
\begin{lemma} \label{lem1} Suppose there exist sets $A \subset\mathcal{Z}$ and $B \subset\mathcal{Z}^M$, $M\geq 1$, and $\epsilon>0$ such that
\begin{align*}
&P(Z_k \in A \mbox{ i.o.})=1,\\
&P(Z_{1:M} \in B|Z_1=z)\geq \epsilon, \quad \forall z \in A.
\end{align*}
Then
\begin{align*}
P(Z \in B \mbox{ i.o.})=1.
\end{align*}
\end{lemma}
\begin{proof}
The proof is just a slightly modified version of the proof of \cite[Th. 9.1.3]{MT}. It suffices to show that
\begin{align} \label{Stat}
P(Z \in A \mbox{ i.o.}) \leq P(Z \in B \mbox{ i.o.}).
\end{align}
Define
\begin{align*}
E_{n}=\{Z_{n:n+M-1} \in B\}, \quad n\geq 1.
\end{align*}
For each $n \geq 1$ let $\mathcal{F}_n$ be a $\sigma$-field generated by $\{Z_1,...,Z_n\}$. First we show that as $n \rightarrow \infty$
\begin{align} \label{c1}
P\left( \bigcup_{i=n}^\infty E_i | \mathcal{F}_n \right) \rightarrow \mathbb{I}\left( \bigcap_{k=1}^\infty \bigcup_{i=k}^\infty E_i \right), \quad \mbox{a.s.},
\end{align}
where $\mathbb{I}(\cdot)$ is the indicator function. To see this, note that for fixed $l \leq n$
\begin{align} \label{ineq2}
P\left( \bigcup_{i=l}^\infty E_i| \mathcal{F}_n \right) \geq P\left( \bigcup_{i=n}^\infty E_i| \mathcal{F}_n \right) \geq P\left(\bigcap_{k=1}^\infty \bigcup_{i=k}^\infty E_i| \mathcal{F}_n \right).
\end{align}
Applying the Martingale Convergence Theorem to the extreme elements of the inequalities (\ref{ineq2}), we obtain
\begin{align} \label{ineq3}
\mathbb{I}(\cup_{i=l}^\infty E_i ) \geq \limsup_n P(\cup_{i=n}^\infty {E_i}|\mathcal{F}_n) \geq \liminf_n P(\cup_{i=n}^\infty {E_i}|\mathcal{F}_n)
\geq \mathbb{I}(\cap_{k=1}^\infty \cup_{i=k}^\infty E_i).
\end{align}
As $l \rightarrow \infty$, the two extremes in (\ref{ineq3}) converge, which shows that the convergence (\ref{c1}) holds as required.

Next, define
\begin{align*}
L(z)=P(\cup_{i=1}^\infty E_i|Z_1=z), \quad z \in \mathcal{Z}.
\end{align*}
Clearly $L(z)\geq \epsilon$ for every $z \in A$. Also by Markov property $P(\cup_{i=n}^\infty E_i| \mathcal{F}_n)=L(Z_n)$ a.s. Thus, using (\ref{c1}), we have almost surely
\begin{align*}
\mathbb{I}(\cap_{k=1}^\infty \cup_{i=k}^\infty \{Z_i \in A\}) & \leq \mathbb{I}(\limsup_n L(Z_n) \geq \epsilon)= \mathbb{I}(\lim_n L(Z_n) =1 )= \mathbb{I}(\cap_{k=1}^\infty \cup_{i=k}^\infty E_i).
\end{align*}
This implies (\ref{Stat}).
\end{proof}

\begin{lemma}\label{HMMHarris} Let $Z$ be HMM and define $G_i=\{x \in \X \:| \: f_i(x)>0\}$ for $i \in \Y$. If Markov chain $Y$ is irreducible then $Z$ is $\psi$-irreducible, where
\begin{align*}
\psi(A \times \{i\})=\mu(A \cap G_i), \quad A \in \B(\X), \quad i \in \Y,
\end{align*}
 and Harris recurrent.
\end{lemma}
\begin{proof} That $Z$ is $\psi$-irreducible follows directly from the irreducibility of $Y$ and definition of HMM. Because for every $x \in \X$ and $A \in \B(\Z)$
\begin{align*}
P(Z_2 \in A|Z_1=(x,1))=P(Z_2 \in A|Y_1=1),
\end{align*}
then set $\X \times \{1\}$ is small  (see definition in \cite{MT}). Harris recurrence now follows from \cite[Prop. 9.1.7]{MT}.
\end{proof}

\section{Proof of Lemma \ref{LMHarris}} \label{LMHarrisProof}
We start with the following auxiliary lemma:
\begin{lemma} \label{lemHarris} Let $Z$ be the linear Markov switching model defined by \eqref{LMSM}. Suppose $Z$ is $\psi$-irreducible, support of $\psi$ has non-empty interior, and there exists a compact set $C \subset \mathbb{R}^d \times \Y$ and a positive measurable function $V \colon \mathcal{Z} \rightarrow \mathbb{R}_{>0}$ satisfying
\begin{align} \label{drift}
\mathbb{E}[V(Z_2)|Z_1=z]-V(z)\leq 0, \quad \forall z \in \mathcal{Z} \setminus C.
\end{align}
If set $\{z \: | \: V(z) \leq k\}$ is contained in a compact set for every $k< \infty$, then $Z$ is Harris recurrent.
\end{lemma}
\begin{proof} According to our assumption space $\Z$ is equipped with product topology $\tau \times 2^\Y$, where $\tau$ is the the topology on the Euclidean space $\X=\mathbb{R}^d$. In this topology saying that some function $h \colon \Z \rightarrow \mathbb{R}$ is continuous means the following: for every $x_0 \in \X$ and $i \in \Y$
\begin{align*}
\lim_{x \rightarrow x_0}h(x,i)=h(x_0,i).
\end{align*}
First we show that for any bounded and continuous function $h \colon \mathcal{Z} \rightarrow \mathbb{R}$, the function \begin{align} \label{Ehmap}
z \mapsto \mathbb{E}[h(Z_2)|Z_1=z]
\end{align}
is also bounded and continuous. Indeed, we have
\begin{align*}
\mathbb{E}[h(Z_2)|Z_1=(x',i)]&=\int_\Z h(x,j) \,  q(x,j|x',i) \mu \times c(d(x,j))\\
&=\int_\Z h(x,j)p_{ij}f_j(x|x') \, \mu \times c(d(x,j))\\
&=\sum_{j \in \Y}\int_\X h(x,j)p_{ij}f_j(x|x') \, \mu(dx)\\
&=\sum_{j \in \Y}\int_\X h(x,j)p_{ij}h_j(x-F(j)x') \, \mu(dx)\\
&=\sum_{j \in \Y}\int_\X h(x+F(j)x',j)p_{ij}h_j(x) \, \mu(dx).
\end{align*}
Thus (\ref{Ehmap}) is bounded, and also continuous by Dominated Convergence Theorem. In what follows, the definitions for the terms in italic can be found from \cite{MT}. By \cite[Prop. 6.1.1(i)]{MT} $Z$ is \textit{weak Feller}. Hence by \cite[Prop. 6.2.8]{MT} every compact set in $\mathcal{Z}$ is \textit{petite}. Thus the statement follows from \cite[Th. 9.1.8]{MT}.
\end{proof}

We take $V(x,i)=\|x\|_1+1$; then
\begin{align*}
\mathbb{E}[V(Z_2)|Z_1=(x',i)]-V(x',i)&=\sum_{j \in \Y}p_{ij}\int \|x\|_1 h_{j}(x-F(j)x') \,  \mu(dx)-\|x'\|_1\\
&=\sum_{j \in \Y}p_{ij}\int \|x+F(j)x'\|_1 h_{j}(x) \,  \mu(dx)-\|x'\|_1\\
&\leq \sum_{j \in \Y}p_{ij}\int ( \|x\|_1+ \|F(j)\|_1\|x'\|_1) h_{j}(x) \,  \mu(dx)-\|x'\|_1\\
&\leq \sum_{j \in \Y}p_{ij}\mathbb{E}\|\xi_2(j)\|_1 +\|x'\|_1  \sum_{j \in \Y}p_{ij}\|F(j)\|_1 -\|x'\|_{1}.
\end{align*}
Thus by the assumptions that the expectations $\mathbb{E}\|\xi(i)\|_1$ are finite and $\max_{i \in \Y}\sum_{j \in \Y}p_{ij}\|F(j)\|_1<1$, we have that (\ref{drift}) holds with $C=[-n,n]^d \times \Y$, when $n$ is sufficiently large. Also set $\{z \: | \: V(z) \leq k\}$ is contained in a compact set $[-k,k]^d \times \Y$ for every $k< \infty$. Hence Lemma \ref{lemHarris} applies.
\end{appendices}
\paragraph{Acknowledgment.} The research is supported by Estonian  institutional research funding
IUT34-5.


\printbibliography
\end{document}